%% file: main.tex
\documentclass[12pt]{article}

\input{header}

\title{On Fourier Phase Retrieval from Differential Intensity Measurements with Applications to Wavefront Sensing}
\author{
Simon Hubmer\footnote{Johannes Kepler University Linz, Institute of Industrial Mathematics, Altenbergerstra{\ss}e~69, 4040 Linz, Austria, (simon.hubmer@jku.at), \textbf{Corresponding author}.} ,
Lukas Weissinger\footnote{Johann Radon Institute Linz, Altenbergerstra{\ss}e~69, 4040 Linz, Austria, (lukas.weissinger@ricam.oeaw.ac.at)} ,
\\
Ronny Ramlau\footnote{Johannes Kepler University Linz, Institute of Industrial Mathematics, Altenbergerstra{\ss}e~69, 4040 Linz, Austria, (ronny.ramlau@jku.at)}\,\,\textsuperscript{,}\footnote{Johann Radon Institute Linz, Altenbergerstra{\ss}e~69, 4040 Linz, Austria, (ronny.ramlau@ricam.oeaw.ac.at)} ,
Otmar Scherzer\footnote{University of Vienna, Faculty of Mathematics, Oskar Morgenstern-Platz 1, 1090 Vienna, Austria (otmar.scherzer@univie.ac.at)}\,\,\textsuperscript{,}\footnote{Johann Radon Institute Linz, Altenbergerstra{\ss}e~69, 4040 Linz, Austria, (otmar.scherzer@assoc.oeaw.ac.at)}
}

\begin{document}

% Include the title
\maketitle

% Abstract
\begin{abstract}
In this paper, we consider Fourier phase retrieval from differential intensity measurements, i.e., the problem of determining the phase of a complex-valued function from a series of intensity measurements differing only by slight modulations in Fourier domain. These modulations may be induced by optical elements such as prisms or phase plates, or via spatial-light modulators. Generalizing the principles behind the transport of intensity equation, we show that given such differential intensity measurements, the phase can be determined as the solution of certain partial differential or integro-differential equations. This is then used to design efficient reconstruction algorithms for a number of Fourier-type wavefront sensors, such as the pyramid wavefront sensor, commonly used in adaptive optics. Numerical experiments illustrate the usefulness of our proposed approach.

\end{abstract}

% % % % % % % % % % % % %
% Start of the sections %
% % % % % % % % % % % % %

% % % % % % % % % % % % % %
% Section - Introduction  %
% % % % % % % % % % % % % %
\section{Introduction}\label{sect_introduction}
    
In this paper, we consider a variant of the classic Fourier phase retrieval problem \cite{Hurt_1989}
    \begin{equation}\label{prob_classic}
        \text{Given} \quad I(\xv) := \abs{\FI (u)(\xv)}^2\,, \quad \text{find} \quad u(\xiv) \qquad (\text{where} \, u : \R^N \to \C ) \,.
    \end{equation}
This problem has been studied extensively in the past, both from a mathematical and an applied perspective \cite{Hurt_1989}. Many of these investigations are motivated by applications in imaging, such as crystallography, microscopy, holography, diffraction imaging, and wavefront sensing in optical systems. This is because there, the complex-valued function
    \begin{equation}\label{polar_form}
        \FI(u )(\xv)= A( \xv ) e^{i \vphi( \xv)} \,,
        \qquad
        \text{with amplitude} 
        \qquad
        A(\xv) = \sqrt{I(\xv)} 
    \end{equation}
corresponds to some electric field, whose phase $\vphi(\xv)$ encodes important physical information of interest. However, physical detectors can typically only measure the intensity $I(\xv) = A(\xv)^2$, while the phase information $\vphi(\xv)$ is lost. Hence, the Fourier phase retrieval problem \eqref{prob_classic} can be equivalently rewritten as the inverse problem 
    \begin{equation}\label{prob_classic_simple}
        \text{Given} \quad A(\xv) \,, \quad \text{find} \quad \vphi(\xiv) \,.
    \end{equation}
Clearly, this is a severely ill-posed problem, for which one can only hope to obtain a unique reconstruction under suitable additional assumptions on the function $u(\xiv)$, such as smoothness and compact support, or knowledge of the location of the zeros of its Fourier-Laplace transform. For an overview of the substantial literature on this topic, including numerical algorithms for the practical solution of \eqref{prob_classic} or \eqref{prob_classic_simple}, see e.g.~\cite{Hurt_1989}.

Due to the severe difficulties inherent in the classic Fourier phase retrieval problem, a popular approach for its solution is to partly circumvent it. More precisely, instead of the single intensity measurement $I(\xv)$ in \eqref{prob_classic}, one often uses several intensity measurements $I_1(\xv),I_2(\xv),\text{etc.}$ which are suitably related. For example, one may consider
    \begin{equation*}
        I_1(\xv) := \abs{\FI (u)(\xv)}^2\,, 
        \qquad
        I_2(\xv) := \abs{\FI (v)(\xv)}^2\,, 
        \qquad
        I_3(\xv) := \abs{\FI (u+v)(\xv)}^2\,, 
    \end{equation*}
where $v$ is some (possibly partly-known) reference function. Then one can show that
    \begin{equation}\label{OCT}
    \begin{split}
        I_3(\xv) &=
        I_1(\xv) + I_2(\xv) 
        + 2 \Re\kl{ \overline{\FI (u)(\xv)}\FI (v)(\xv)  }
        \\
        &= I_1(\xv) + I_2(\xv)  + 2 \sqrt{I_1(\xv)}  \sqrt{I_2(\xv)} \cos(\vphi_v(\xv) - \vphi_u(\xv)) \,,
    \end{split}
    \end{equation} 
where $\vphi_u(\xv)$, $\vphi_v(\xv)$ denote the phases of $\FI(u )(\xv)$, $\FI(v)(\xv)$ as in \eqref{polar_form}, respectively. Hence, the relative phase difference $(\vphi_v(\xv) - \vphi_u(\xv))$ may be recovered from the intensities $I_1(\xv),I_2(\xv), I_3(\xv)$, which forms the basis of interferometry and holography.

Another popular approach is windowed or masked Fourier phase retrieval, where
    \begin{equation*}
        I_k(\xv) := \abs{\FI ( m_k u)(\xv)}^2 \,,
    \end{equation*}
and $\Kl{m_k}_{k\in K}$ is a finite or infinite sequence of mask functions. In this setting, unique phase retrieval up to a constant phase factor is possible with as little as three different masks. However, these masks are often real-valued, making them difficult or sometimes impossible to implement physically within an optical system. Furthermore, simultaneously recording several intensity images can be limiting in practice, since this requires beam-splitting, leading to reduced exposure and thus limiting imaging quality.

Hence, in this paper, we consider intensity measurements of the form
    \begin{equation}\label{def_Ixt}
        I(\xv,t) := \abs{\FI (e^{i\psi(t)} u)(\xv)}^2 \,,
    \end{equation}
where $\psi(t) = \psi(\xiv,t)$ is a parameter or time-dependent function, and the problem 
    \begin{equation*}
        \text{Given} \quad I(\xv,t) \,, \quad \text{find} \quad u(\xiv) \,.
    \end{equation*}
Working with the measurements $I(\xv,t)$ has two motivations: biplane or vortex-plate imaging in microscopy, and wavefront sensing via Fourier-type wavefront sensors (WFSs). In particular, the modulation $e^{i \psi(t)}$ may be realized physically in several different ways, e.g.\ using lenses, phase plates, optical prisms, or spatial light modulators (SLMs). Especially SLMs provide for the greatest flexibility here, since they allow almost any real-valued function $\psi$ to be realized in practice. Additionally, they may be smoothly varied over time, allowing us to consider so-called differential intensity measurements
    \begin{equation*}
        \ddt I(\xv,t) = \ddt \abs{\FI (e^{i\psi(t)} u)(\xv)}^2 \,.
    \end{equation*}
This leads to the problem of Fourier phase retrieval from differential measurements,~i.e.,
    \begin{equation}\label{prob_differential}
        \text{Given} \quad I(\xv,t) \,\,\, \text{and} \,\,\, \ddt  I(\xv,t) \,, \quad \text{find} \quad u(\xiv) \,.
    \end{equation}
Perhaps the most well-known instance of this problem occurs for $\psi(t) = C\, t \abs{\xiv}^2$, which can physically be realized, e.g., by defocusing the imaging system, and is thus connected to the concept of biplane (or multi-plane) imaging. In this case, there holds
    \begin{equation}\label{TIE_basic}
        c \, \ddt I(\xv,t) = \nabla_{\xv} \cdot \kl{I(\xv,t) \nabla_{\xv} \vphi(\xv,t)} \,,
    \end{equation}
where the constant $c$ depends on the scaling convention of the Fourier transform, cf.~Section~\ref{subsect_defoucs}, and the phase $\vphi(\xv,t)$ is defined, analogously to \eqref{polar_form}, via the decomposition
    \begin{equation}\label{polar_form_t}
        \FI(e^{i\psi(t)} u )(\xv)= A( \xv,t ) e^{i \vphi( \xv,t)} \,,
        \qquad
        \text{with amplitude} 
        \qquad
        A(\xv,t) = \sqrt{I(\xv,t)} \,.
    \end{equation}
The PDE \eqref{TIE_basic}, evaluated at $t=0$, is a variant of the transport of intensity (TIE) equation, which has been used in phase retrieval for a long time \cite{Gonsalves_1987,Zuo_2020_TIE_Tutorial}. Note that in practice, the time derivative is usually approximated via the difference quotient
    \begin{equation*}
        \ddt I(\xv,t) \Big\vert_{t=0} \approx \frac{I(\xv,\tau) - I(\xv,0)}{\tau} \,,
    \end{equation*}
for some $\tau > 0$ sufficiently small. Since $\psi(t) = C \, t \abs{\xiv}^2$, which corresponds to a defocus, the two intensity images $I(\xv,\tau)$ and $I(\xv,0)$ can be obtained by recording a defocused and an in-focus image, respectively. Note that this can be realized both sequentially, as indicated by $t$ denoting ``time'', but also simultaneously via the use of a beam-splitter.

In this paper, we generalize the key ideas behind the TIE to derive explicit equations for the phase $\vphi(\xv,t)$ for large classes of modulation functions $\psi(t)$. These include general polynomial modulations, as well as modulations of power and step-function type, which in turn lead to differential and integro-differential equations involving fractional Laplace and Hilbert transforms, among others. The obtained equations can then be used to (numerically) solve the differential Fourier phase retrieval problem \eqref{prob_differential}. The analyzed classes of modulation functions $\psi(t)$ are in particular motivated by so-called Fourier-type wavefront sensors (WFSs), which can mathematically be modeled by \cite{HuNeuSha_2023}
    \begin{equation}\label{model_FWFS}
        I(\xv) := \abs{\FI\kl{e^{i\psi} \F\kl{ \chi_A e^{i\phi}}  }(\xv)}^2 \,.
    \end{equation}
Here, $\chi_A$ denotes the indicator function of the aperture of the imaging system, $\phi$ is the unknown wavefront aberration to be recovered, $\psi$ encodes the shape of the optical element used in the Fourier-type WFS, and $I(\xv)$ is the measured intensity pattern. While the shape function $\psi$ is typically not varied during the measurement process, the structure of the Fourier-type WFS model allows us to nevertheless apply our derived results and explicit equations from the setting of \eqref{def_Ixt}. This then leads both to new reconstruction approaches for Fourier-type WFS, and a new foundation of existing ones.

The structure of this paper is as follows: In Section~\ref{sect_background}, we recall the necessary mathematical background which we use in Section~\ref{sect_phase_diff} to derive explicit formulas for phase reconstruction from differential intensity measurements. In Section~\ref{sect_WFS}, we then apply our derived results to Fourier-type WFSs, and numerically test the corresponding reconstruction approaches in Section~\ref{sect_numerics}, before ending with a short conclusion in Section~\ref{sect_conclusion}.

% % % % % % % % % % % % % % % % % % %
% Section - Mathematical Background %
% % % % % % % % % % % % % % % % % % %
\section{Mathematical Background}\label{sect_background}

In this section, we recall some necessary mathematical background, in particular on Fourier and Hilbert transforms, collected from standard literature such as \cite{McLean_2000,Evans_1998}. First, recall the classical multiindex notation: An $N$-tuple $\alpha = \kl{\alpha_1\,, \dots \,, \alpha_N}$ of non-negative integers is called a multiindex, and $\abs{\alpha} := \alpha_1 + \dots + \alpha_N$ is called the order of $\alpha$. Furthermore, for two multiindices $\alpha$ and $\beta$, we define
	\begin{equation*}
		\binom{\alpha}{\beta} := \frac{\alpha!}{\beta!(\alpha-\beta)!} \,,
		\qquad \text{where} \qquad
		\alpha! := \alpha_1! \cdots \alpha_N! \,.
	\end{equation*} 
Next, for any integrable $u : \RN \to \C$, we consider the Fourier transform 
	\begin{equation*}
		(\F u)(\xiv) := \int_\RN e^{-2\pi i \xiv \cdot \xv} u(\xv) \, d \xv \,,
		\qquad
		\forall \, \xiv = (\xi_1\,,\dots\,,\xi_N)  \in \RN \,,
	\end{equation*}
as well as the inverse Fourier transform
	\begin{equation*}
		(\FI u)(\xv) = \int_\RN e^{2\pi i \xv \cdot \xiv} u(\xiv) \, d\xiv \,, 
		\qquad
		\forall \, \xv = (x_1\,,\dots\,,x_N)\in \RN \,.
	\end{equation*}
Both operators $\F$ and $\FI$ uniquely determine bounded linear operators
	\begin{equation*}
		\F : \LtRN \to \LtRN \,,
		\qquad
		\FI : \LtRN \to \LtRN \,,
	\end{equation*}
satisfying $\F \FI = \FI \F = I$, cf.~\cite{McLean_2000,Evans_1998}. Furthermore, together with the definition
	\begin{equation*}
		D^\alpha u := \frac{\partial^{\abs{\alpha}}}{\partial x_1 ^{\alpha_1} \, \cdots \, \partial x_N^{\alpha_N}} u 
		=
		\partial_{x_1}^{\alpha_1} \, \cdots \, \partial_{x_N}^{\alpha_N} u \,,
	\end{equation*}
it follows that for any sufficiently rapidly decaying function $u$ there holds
	\begin{equation}\label{Fourier_diff}
		D^\alpha \FI( u )(\xv) = \FI\kl{\kl{2\pi i \xiv}^\alpha u(\xiv)}(\xv) \,.
	\end{equation}
Finally, the Fourier transform satisfies the convolution property
	\begin{equation}\label{Fourier_convolution}
		\F( u \ast v ) = \F(u) \F(v) \,,
		\qquad
		\text{where}
		\qquad
		(u \ast v)(\xv) := \int_\Rt u(\xv-\yv) v(\yv) \, d\yv \,.
	\end{equation}
The above definitions and properties can also be extended to a distributional setting. 

Next, consider the Hilbert transform $\H$, which for $u \in L_2(\R)$ is defined by
    \begin{equation*}
        \mathcal{H}(u)(x) := \frac{1}{\pi}\, \operatorname{p.v.} \int_\R\frac{u(z)}{x-z} \, dz \,,
    \end{equation*}
where p.v.\ denotes the Cauchy principal value \cite{King_2009}. It can be shown that  
    \begin{equation*}
        (\H u)(x) = \kl{u(z) \ast \frac{1}{\pi z}}(x) = \FI\kl{ -i \,\sgn (\xi) \F(u)(\xi)}(x)\,,
    \end{equation*}
from which it follows that $\HI = - \H$. More generally, for $u \in \LtRN$, we have
    \begin{equation}\label{eq_Hk_sgn}
        \FI\kl{-i \sgn(\xi_k) \F(u)(\xiv) }(\xv)
        =
        (\Hk u)(\xv) \,,
    \end{equation}
where 
    \begin{equation*}
            (\Hk u)(\xv) := \frac{1}{\pi}\, \operatorname{p.v.} \int_\R\frac{u(x_1,\ldots,x_{k-1},z,x_{k+1},\ldots,x_N)}{x_k-z} \, d z
    \end{equation*}
denotes the one-dimensional Hilbert transform with respect to the $k$-th variable.

Finally, we also require the fractional Laplacian, which for $s \in \R_0^+$ is defined by
    \begin{equation}\label{def_fracLapl}
        (-\Delta)^{s} u(\xv) := \FI\kl{\abs{2\pi \xiv}^{2s} \F(u)(\xiv)  } (\xv) \,,
    \end{equation}
generalizing \eqref{Fourier_diff} to fractional differentiation orders. Note that for $s=1/2$, we have
    \begin{equation}\label{def_fracLaploh}
        (-\Delta)^{\frac{1}{2}} u(\xv) = \FI\kl{\abs{2\pi \xiv} \F(u)(\xiv)  } (\xv) \,,
    \end{equation}
and thus, in the one-dimensional case $u \in \LtR$, where $\abs{\xiv} = \abs{\xi} =  \sgn(\xi)\xi$, we obtain
    \begin{equation}\label{eq_Lapoh_Hilbert}
        (-\Delta)^{\frac{1}{2}} u(x) = \FI\kl{2\pi \,\sgn(\xi)\xi  \F(u)(\xi)  } (x) 
        =
        (\H \circ \partial_x) u(x)
        =
        (\partial_x \circ \H) u(x)\,,
    \end{equation}
which connects the fractional Laplacian with the Hilbert transform. For an analogous connection in higher dimensions, i.e., for $u \in \LtRN$, one considers the Riesz transforms
    \begin{equation}\label{def_Riesz}
        (\Rk u)(\xv) := \FI\kl{- i \frac{\xi_k}{\abs{\xiv}} \F(u)(\xiv)  } (\xv) 
        \,,
    \end{equation}
which can also be expressed via principal value integrals. Since due to \eqref{def_fracLaploh}, there holds
    \begin{equation*}
        (-\Delta)^{\frac{1}{2}} u(\xv) = \sum_{k=1}^N  \FI\kl{2\pi \frac{\xi_k^2}{\abs{ \xiv} } \F(u)(\xiv)  } (\xv) 
        =
        \sum_{k=1}^N  \FI\kl{- i\frac{\xi_k}{\abs{ \xiv} }(2 \pi i \xi_k) \F(u)(\xiv)  } (\xv) \,,
    \end{equation*}
it follows with the vectorial Riesz transform $\mathcal{R}:=(\mathcal{R}_{x_1},\dots,\mathcal{R}_{x_N})$ that
    \begin{equation}\label{helper_Riesz}
        (-\Delta)^{\frac{1}{2}}  u(\xv)
        = \sum_{k=1}^N (\Rk \circ \partial_{x_k}) u(\xv)
        = \sum_{k=1}^N (\partial_{x_k} \circ \Rk) u(\xv)
        = \div\kl{(\mathcal{R} u)(\xv)} \,,
    \end{equation}
which generalized \eqref{eq_Lapoh_Hilbert} to higher dimensions by connecting the fractional Laplacian and the Riesz transforms. Note that in the 1D case $u \in \LtR$, there holds $\mathcal{R} u = \H u$.

% % % % % % % % % % % % % % % % % % % % % % % % % % % % % % %
% Section - Phase Retrieval from Differential Measurements  %
% % % % % % % % % % % % % % % % % % % % % % % % % % % % % % %
\section{Phase Retrieval from Differential Measurements}\label{sect_phase_diff}

In this section, we consider the differential Fourier phase retrieval problem \eqref{prob_differential}, i.e.,
    \begin{equation*}
        \text{Given} \quad I(\xv,t) \,\,\, \text{and} \,\,\, \ddt  I(\xv,t) \,, \quad \text{find} \quad u(\xiv) \qquad (u : \R^N \to \C )  \,,
    \end{equation*}
where the time-dependent intensity measurements $I(\xv,t)$ are defined by \eqref{def_Ixt}, i.e.,
    \begin{equation*}
        I(\xv,t) := \abs{\FI (e^{i\psi(t)} u)(\xv)}^2 \,,
    \end{equation*}
with a known function $\psi(t) = \psi(\xiv,t)$. To do so, we first recall the polar form \eqref{polar_form_t}, i.e.,
    \begin{equation*}
        \FI(e^{i\psi(t)} u )(\xv)= A( \xv,t ) e^{i \vphi( \xv,t)} \,,
        \qquad
        \text{with amplitude} 
        \qquad
        A(\xv,t) := \sqrt{I(\xv,t)} \,,
    \end{equation*}
for some unknown phase function $\vphi(\xv,t)$. Note that since this can be rewritten as
    \begin{equation*}
        u (\xiv)= e^{-i\psi(t)}\F\kl{A( \xv,t ) e^{i \vphi( \xv,t)}}(\xiv) \,, 
        \qquad
        \forall \, t \in [0,T] \,,
    \end{equation*}
and since $A(\xv,t) = \sqrt{I(\xv,t)}$ is assumed to be known, the differential Fourier phase retrieval problem \eqref{prob_differential} can be solved by computing the phase $\vphi(\xv,\tau)$ for a single value of $\tau \in [0,T]$. Now, throughout the subsequent analysis, we make the following

\begin{assumption}\label{ass_main}
The functions $u : \RN \to \C$ and $\vphi\,,\psi\,,A : \RN \times [0,T]  \to \R$ for some $T > 0$ satisfy the following properties: $u \in \LtRN$ and $\vphi(\cdot,t)\,,A(\cdot,t) \in \LtRN$ for each $t \in [0,T]$. Furthermore, $\psi(\xiv,\cdot) \in C^1([0,T])$, and $\psi'(t) := \ddt \psi (t) = \ddt \psi (\xiv, t)$. 
\end{assumption}

Note that under the above assumption, the intensity function $I(\xv,t)$ is well-defined for each $t \in [0,T]$. Now, going back to \cite{Gonsalves_1987}, the key idea of phase retrieval from differential measurements is to consider the time derivative of the intensity function $I(\xv,t)$, i.e.,
	\begin{equation}\label{difference_quotient}
		\ddt I(\xv,t) 
		= \lim_{h \to 0} \frac{I(\xv,t+h) - I(\xv,t)}{h}
		\approx \frac{I(\xv,t+\tau) - I(\xv,t)}{\tau} \,,
	\end{equation}
where the last approximation is valid for $\tau \approx 0$. Note that this approximation can be used to numerically compute $\ddt I(\xv,t)$ given suitable samples of $I(\xv,t)$. To proceed further, we now compute an explicit expression for $\ddt I(\xv,t)$ in the following  

\begin{theorem}\label{thm_ddt_I_general}
Let $I(\xv,t)$ be defined as in \eqref{def_Ixt}, let Assumption~\ref{ass_main} hold, and assume that for each $t \in [0,T]$ there holds $u\psi'(t) \in \LtRN$. Then $I(\xv,t)$ is continuously differentiable with respect to $t$, i.e., $I(\xv,t) \in C^1([0,T],L^1(\RN))$, and there holds
	\begin{equation}\label{eq_ddt_I_general}
		\ddt I(\xv,t) = 2 \Re\kl{ \overline{\FI(e^{i\psi(t)}u)(\xv)} \FI(e^{i\psi(t)} u i \psi'(t) )(\xv) } \,.
	\end{equation}
Furthermore, if $A(\xv,t)$, $\vphi(\xv,t)$ as defined in \eqref{polar_form_t} are also differentiable wrt.\ $t$, then
	\begin{equation}\label{eq_ddt_phi_general}
		2 I(\xv,t) \ddt \vphi(\xv,t) = \Im\kl{ \overline{\FI(e^{i\psi(t)}u)(\xv)} \FI(e^{i\psi(t)} u i \psi'(t) )(\xv) } \,.
	\end{equation}
\end{theorem}
\begin{proof}
The first part of the theorem follows from the chain rule applied to \eqref{def_Ixt} as well as the regularity assumptions on the functions $u$ and $\psi$. For the second part, consider
    \begin{equation*}
        w(\xv, t):= \FI(e^{i \psi(t)}u)(\xv) \,,
        \qquad
        \text{for which}
        \qquad 
        \ddt w(\xv, t) = \FI(e^{i \psi(t)} i \psi'(t) u)(\xv) \,.
    \end{equation*}
Together with the polar form \eqref{polar_form_t}, we obtain that
    \begin{equation*}
    \begin{split}
        & 2 \Im \kl{\overline{w(\xv,t)} \ddt w(\xv,t)} 
        = 
        2 \Im\kl{ A( \xv, t) e^{-i \vphi( \xv,t)} \ddt \kl{A(\xv,t)e^{i \vphi(\xv,t)}}}
        \\
        & \qquad 
        = 2 \Im\kl{A( \xv, t) e^{-i \vphi( \xv,t)}  \kl{\frac{\partial A(\xv,t)}{\partial t}e^{i \vphi( \xv,t)}+ i A(\xv,t) \frac{\partial \vphi(\xv,t)}{\partial t}e^{i \vphi( \xv,t)}}}
        \\
        & \qquad 
        = 2 A(\xv,t)^2 \frac{\partial }{\partial t}\vphi(\xv,t)
        = 
        2 I( \xv,t) \ddt \vphi( \xv,t)
    \end{split}
    \end{equation*}
which directly yields the assertion.
\end{proof}

The above result, and in particular \eqref{eq_ddt_I_general}, forms the starting point of our subsequent analysis and derivation of (integro-)differential equations for the phase $\vphi(\xv,t)$. Note first by using the convolution property \eqref{Fourier_convolution}, the right hand side of \eqref{eq_ddt_I_general} takes the form
	\begin{equation*}%\label{interference_conv}
		2 \Re\kl{ \overline{\FI(e^{i\psi(t)}u)(\xv)} \kl{\FI(i \psi'(t) ) \ast \FI(e^{i\psi(t)} u)}(\xv) } \,,
	\end{equation*}
which together with the polar decomposition \eqref{polar_form_t} can be rewritten as
    \begin{equation}\label{interference_conv_polar}
		2 \Re\kl{ A(\xv,t)e^{-i\vphi(\xv,t)} \kl{\FI(i \psi'(t) ) \ast \kl{A(\cdot,t)e^{-i\vphi(\cdot,t)}} }(\xv) } \,.
	\end{equation}
This is reminiscent of the interference term $2 \Re(\overline{\FI (u)(\xv)}\FI (v)(\xv) )$ in \eqref{OCT}, where now $v(\xiv)$ is a modified version of the unknown $u(\xiv)$ interfering with $u(\xiv)$ itself. Now the key idea, which goes back to arguments from \cite{Gonsalves_1987}, is that for suitably chosen functions $\psi(t)$, the interference term \eqref{interference_conv_polar} may take on a simpler form, such that then \eqref{eq_ddt_I_general}, i.e.,
    \begin{equation*}
		\ddt I(\xv,t) = 2 \Re\kl{ A(\xv,t)e^{-i\vphi(\xv,t)} \kl{\FI(i \psi'(t) ) \ast \kl{A(\cdot,t)e^{-i\vphi(\cdot,t)}} }(\xv) }  \,,
	\end{equation*}
turns into an (integro-)differential equation for the unknown phase $\vphi(\xv,t)$ which is uniquely solvable (given suitable boundary conditions). That this is indeed the case is the content of the subsequent sections, which generalize and extend arguments from \cite{Gonsalves_1987}.

% % % % % % % % % % % % % % % % % % % %
% Subsection - Polynomial Modulations %
% % % % % % % % % % % % % % % % % % % %
\subsection{Polynomial Modulations}

Inspired by \cite{Gonsalves_1987}, we first consider modulation functions $\psi(\xiv,t)$ of polynomial type.

\begin{lemma}\label{lem_polynomial_pseudo}
Let $u \in \LtRN$, let $\psi_0 \in \LtRN$ be independent of $t$, and define
	\begin{equation}\label{psi_poly}
		\psi(\xiv,t) := \psi_0(\xiv) +  \sum\limits_{\abs{\alpha} \leq M} \ca(t) \, (2\pi)^\alpha \xiv^\alpha \,, 
		\qquad \forall \, t \in [0,T]\,, \xiv \in \RN \,, 
	\end{equation}
where $M \in \N$ and $\ca : [0,T] \to \R$  are continuously differentiable coefficient functions. Furthermore, assume that for each $t \in [0,T]$ there holds $u\psi'(t) \in \LtRN$, i.e., 
	\begin{equation}\label{eq_u_poly_Lt}
		\int_{\RN} \abs{\sum_{\abs{\alpha}\leq M} \ca'(t)\, \xiv^\alpha \, u(\xiv)}^2 \, d\xiv  < \infty \,.
	\end{equation}
Then there holds
    \begin{equation}\label{eq_pseudo_diff}
        \FI(e^{i\psi(t)} u i \psi'(t) )(\xv) = \Pcat \kl{ \FI\kl{e^{i\psi(t)} u }}(\xv) \,,
    \end{equation}
where the (pseudo-)differential operator $\Pcat$ is defined by
	\begin{equation}\label{def_Pca}
		\Pcat:= \sum\limits_{\abs{\alpha} \leq M} (-i)^{\abs{\alpha}-1} \ca'(t) D^\alpha \,.
	\end{equation} 
\end{lemma}
\begin{proof}
Note first that from the definition of $\psi(\xiv,t)$ it follows that
	\begin{equation*}
		\psi'(t) = \ddt \psi(\xiv,t) =  \sum\limits_{\abs{\alpha} \leq M} \ca'(t) \, (2\pi)^\alpha \xiv^\alpha  \,.
	\end{equation*}	
Hence, it follows together with \eqref{Fourier_diff} that
	\begin{equation*}
	\begin{split}
		&\FI(e^{i\psi(t)} u i \psi'(t) )(\xv) 
		= 
        \FI\kl{e^{i\psi(t)} u i \sum\limits_{\abs{\alpha} \leq M} \ca'(t) \, (2\pi)^\alpha \xiv^\alpha  }(\xv) 
        \\
        & 
        \qquad =
		\sum\limits_{\abs{\alpha} \leq M} (-i)^{\abs{\alpha}-1} \ca'(t) \FI\kl{(2\pi i \xiv)^\alpha e^{i\psi(t)} u }(\xv)
		\\
		& 
        \qquad =
		\sum\limits_{\abs{\alpha} \leq M} (-i)^{\abs{\alpha}-1} \ca'(t)  D^\alpha \FI\kl{e^{i\psi(t)} u }(\xv) \,,
	\end{split}
	\end{equation*}
which together with the definition \eqref{def_Pca} of $\Pcat$ now yields the assertion. 
\end{proof}

\begin{corollary}
Let Assumption~\ref{ass_main} and the assumptions of Lemma~\ref{lem_polynomial_pseudo} hold. Then for $I(\xv,t)$ as in \eqref{def_Ixt}, and the amplitude $A(\xv,t)$ and phase $\vphi(\xv,t)$ as in \eqref{polar_form_t}, there holds
	\begin{equation}\label{eq_ddt_I_poly_diff}
		\ddt I(\xv,t) = 2 \Re\kl{ A(\xv,t)e^{-i\vphi(\xv,t)} \, \Pcat \kl{A(\xv,t)e^{i\vphi(\xv,t)}} } \,,
	\end{equation}
and, if $A(\xv,t)$ and $\vphi(\xv,t)$ are also differentiable wrt.\ $t$, then
    \begin{equation}\label{eq_ddt_phi_poly_diff}
		2 I(\xv,t)\ddt \vphi(\xv,t)  =  \Im\kl{ A(\xv,t)e^{-i\vphi(\xv,t)} \, \Pcat \kl{A(\xv,t)e^{i\vphi(\xv,t)}} } \,.
	\end{equation}
\end{corollary}
\begin{proof}
This follows by combining \eqref{eq_pseudo_diff} and \eqref{polar_form_t} with \eqref{eq_ddt_I_general} and \eqref{eq_ddt_phi_general}, respectively.
\end{proof}	

Next, note that the assumption $u\psi'(t) \in \LtRN$ guarantees the well-definedness of 
    \begin{equation*}
        \FI(e^{i\psi(t)} u i \psi'(t) )(\xv) = \Pcat \kl{ \FI\kl{e^{i\psi(t)} u }}(\xv) 
        = \Pcat \kl{A(\xv,t)e^{i\vphi(\xv,t)}} \,,
    \end{equation*}
Since $\Pcat$ is a differential operator of order (at most) $M$, and assuming that $A(\xv,t)$ and $\vphi(\xv,t)$ are sufficiently smooth, e.g., $A(\cdot,t) \,, \vphi(\cdot,t) \in H^M(\RN)$, it can be seen that
	\begin{equation}\label{eq_Pcat_chain}
		\Pcat \kl{A(\xv,t) e^{i \vphi(\xv,t)} }(\xv) = e^{i \vphi(\xv,t)} \kl{ \Lcat(A,\vphi)(\xv,t) + i \Scat(A,\vphi)(\xv,t)} \,,
	\end{equation}
where $\Lcat(A,\vphi)(\xv,t)$ and $\Scat(A,\vphi)(\xv,t)$ are real-valued functions depending only on $c_\alpha(t)$ and spatial derivatives (up to order $M$) of $A(\xv,t)$ and $\vphi(\xv,t)$, respectively. With this, we can now state our first main result: 

\begin{theorem}\label{thm_main_polynomials}
Let Assumption~\ref{ass_main} and the assumptions of Lemma~\ref{lem_polynomial_pseudo} hold, and let the intensity $I(\xv,t)$ be as in \eqref{def_Ixt}. Furthermore, let the amplitude $A(\xv,t)$ and the phase $\vphi(\xv,t)$ defined via \eqref{polar_form_t} be sufficiently smooth, e.g., $A(\cdot,t) \,, \vphi(\cdot,t) \in H^M(\Omega)$ for all $t \in [0,T]$. Then with $\Lcat(A,\vphi)(\xv,t)$ and $\Scat(A,\vphi)(\xv,t)$ as in \eqref{eq_Pcat_chain} it follows that
    \begin{equation}\label{eq_ddt_I_poly_diff_L}
		\ddt I(\xv,t) =  2 A(\xv,t) \Lcat(A,\vphi)(\xv,t) \,,
	\end{equation}
and, if $A(\xv,t)$ and $\vphi(\xv,t)$ are also differentiable wrt.\ $t$, then
    \begin{equation}\label{eq_ddt_phi_poly_diff_S}
		2 I(\xv,t)\ddt \vphi(\xv,t)  =   2 A(\xv,t) \Scat(A,\vphi)(\xv,t) \,.
	\end{equation}
\end{theorem}
\begin{proof}
First, note that due to the decomposition \eqref{eq_Pcat_chain} there holds
    \begin{equation*}
        A(\xv,t)e^{-i\vphi(\xv,t)} \, \Pcat \kl{A(\xv,t)e^{i\vphi(\xv,t)}} 
        = 
        A(\xv,t) \kl{ \Lcat(A,\vphi)(\xv,t) + i \,\Scat(A,\vphi)(\xv,t)} \,.
    \end{equation*}
Inserting this into \eqref{eq_ddt_I_poly_diff} and \eqref{eq_ddt_phi_poly_diff}, and noting that both $A(\xv,t)$, $\Lcat(A,\vphi)(\xv,t)$, and $\Scat(A,\vphi)(\xv,t)$ are real-valued functions now yields the assertion.
\end{proof}

Due to the above result, if $\psi(t)$ is of the form \eqref{psi_poly}, and since $A(\xv,t) = \sqrt{I(\xv,t)}$ and $\ddt I(\xv,t)$ are assumed to be known, the unknown phase $\vphi(\xv,t)$ can be determined as the solution of the real-valued, linear PDE \eqref{eq_ddt_I_poly_diff_L}. To illustrate this, we now consider a common, specific choice of $\psi(t)$ connected to biplane imaging and the TIE.

% % % % % % % % % % % % % % % % % % % %
% Subsection - Defocus, Biplane, TIE  %
% % % % % % % % % % % % % % % % % % % %
\subsection{Quadratic Modulations and Transport of Intensity}\label{subsect_defoucs}

In this section, we consider the case that the function $\psi(t)$ in \eqref{psi_poly} has the specific form
	\begin{equation}\label{def_psi_defocus}
		\psi(\xiv,t) := \psi_0(\xiv) +  \sum\limits_{k=1}^N c_k(t) \, (2\pi)^2 \xi_k^2 \,, 
		\qquad \forall \, t \in [0,T]\,, \xiv = (\xi_1\,,\dots\,,\xi_N) \in \RN \,,
	\end{equation}
for some coefficients $c_k \in C^1([0,T])$. In case that $\psi_0 = 0$ and $c_k(t) = ct/2$, we have 
    \begin{equation*}
        \psi(\xiv,t) = c t (2\pi^2 )\abs{ \xiv}^2 \,,
        \qquad
        \text{and thus}
        \qquad
        I(\xv,t) = \abs{\FI (e^{i c t  (2\pi^2) \abs{ \xiv}^2} u(\xiv))(\xv)}^2 \,.
    \end{equation*}
Physically, this choice of $\psi(\xiv,t)$ can be interpreted as a defocus, and correspondingly, the functions $I(\xv,t)$ can be seen as successively defocused intensities. In applications, this type of measurement is typically encountered in biplane or multi-plane imaging, where in the simplest case of biplane imaging, only two measurements are made: an in-focus measurement corresponding to $I(\xv,0)$ and a defocus measurement $I(\xv,\tau)$ for some $\tau > 0$. Note that when using the difference quotient \eqref{difference_quotient}, these two measurements are sufficient to compute an approximation of $\ddt I(\xv,t) \vert_{t=0}$, which as we shall see below is enough to recover the phase $\vphi(\xv,t)$ and thus $u(\xiv)$ via the TIE. However, first we consider the general quadratic case \eqref{def_psi_defocus} and derive the following

\begin{proposition}
For $\psi(\xiv,t)$ as in \eqref{def_psi_defocus} the operator $\Pcat$ as in \eqref{def_Pca} takes the form
    \begin{equation}\label{def_Pck}
		\Pcat = \Pckt:= -i \sum\limits_{k=1}^N  c_k'(t) \frac{\partial^2}{\partial x_k^2} \,.
	\end{equation} 	
Furthermore, the functions $\Lcat(A,\vphi)(\xv,t)$ and $\Scat(A,\vphi)(\xv,t)$ defined in \eqref{eq_Pcat_chain} read
    \begin{equation}\label{eq_LSc_defocus}
    \begin{split}
        \Lcat(A,\vphi)(\xv,t) &=  \sum\limits_{k=1}^N  c_k'(t)\kl{2\ddxk \vphi(\xv,t) \ddxk A (\xv,t) + A(\xv,t) \dtdxkt \vphi(\xv,t) }
        \,,
        \\
        \Scat(A,\vphi)(\xv,t) &=  \sum\limits_{k=1}^N  c_k'(t) \kl{ \kl{\ddxk \vphi(\xv,t)}^2 A(\xv,t) -  \dtdxkt A(\xv,t) } \,,
    \end{split}
    \end{equation}
assuming that $A(\xv,t)$ and $ \vphi(\xv,t)$ are sufficiently smooth (e.g., $A(\cdot,t)\,, \vphi(\cdot,t) \in H^2(\RN)$).
\end{proposition}
\begin{proof}
Note first that the form \eqref{def_Pck} follows directly from the definitions of $\Pcat$ and $\psi(\xiv,t)$. Next, in order to compute $\Lcat(A,\vphi)(\xv,t)$ and $\Scat(A,\vphi)(\xv,t)$, consider
    \begin{equation*}
		\Pckt \kl{A(\xv,t) e^{i \vphi(\xv,t)} }(\xv)  = -i \sum\limits_{k=1}^N  c_k'(t) \frac{\partial^2}{\partial x_k^2} \kl{A(\xv,t) e^{i \vphi(\xv,t)} }   \,.
	\end{equation*}
Since due to the chain rule, for all $1\leq k \leq N$ there holds
	\begin{equation*}
	\begin{split}
		& \frac{\partial^2}{\partial x_k^2} \kl{A(\xv,t) e^{i\vphi(\xv,t)} }
		=
		\frac{\partial}{\partial x_k} \kl{ A_{x_k}(\xv,t) e^{i\vphi(\xv,t)} + i \vphi_{x_k} (\xv,t) A(\xv,t) e^{i\vphi(\xv,t)} }
		\\
		& \qquad = 
		A_{x_k,x_k}(\xv,t) e^{i\vphi(\xv,t)} 
		+ i\vphi_{x_k}(\xv,t) A_{x_k}(\xv,t) e^{i\vphi(\xv,t)}
		+ i \vphi_{x_k,x_k}(\xv,t) A(\xv,t) e^{i\vphi(\xv,t)}  
		\\
		& \qquad \quad  
		+ i \vphi_{x_k}(\xv,t)  \kl{  A_{x_k}(\xv,t) e^{i\vphi(\xv,t)} +i\vphi_{x_k}(\xv,t) A(\xv,t) e^{i\vphi(\xv,t)} } \,,
	\end{split}
	\end{equation*}
where a subscript $x_k$ denotes a derivative with respect to $x_k$, we obtain that
    \begin{equation*}
    \begin{split}
        & \Pckt \kl{A(\xv,t) e^{i \vphi(\xv,t)} }(\xv) =  -i \sum\limits_{k=1}^N  c_k'(t) \frac{\partial^2}{\partial x_k^2} \kl{A(\xv,t) e^{i \vphi(\xv,t)} } 
        \\
        &\qquad =
        e^{i\vphi(\xv,t)} \sum\limits_{k=1}^N  c_k'(t)\kl{2\vphi_{x_k}(\xv,t) A_{x_k}(\xv,t) + \vphi_{x_k,x_k}(\xv,t) A(\xv,t) }
		\\
		& \qquad \qquad
		+ i e^{i\vphi(\xv,t)}  \sum\limits_{k=1}^N  c_k'(t) \kl{ \vphi_{x_k}(\xv,t)^2 A(\xv,t) -  A_{x_k,x_k}(\xv,t) }\,,
    \end{split}
    \end{equation*}
which, after rearranging the first sum into divergence form, noting that $A(\xv,t)$ and $\vphi(\xv,t)$ are real valued functions, and comparing with \eqref{def_Pca} now yields the assertion.
\end{proof}

With this, we can now derive two PDEs for the unknown phase $\vphi(\xv,t)$ in

\begin{theorem}\label{thm_TIE_general}
Let $u \in \LtRN$, $\psi_0 \in \LtRN$ be independent of $t$, $\psi$ be as in \eqref{def_psi_defocus} for some coefficients $c_k \in C^1([0,T])$, and let $u\psi'(t) \in \LtRN$ for each $t \in [0,T]$. Furthermore, let $I(\xv,t)$ be as in \eqref{def_Ixt}, $A(\xv,t)$ and $\vphi(\xv,t)$ defined via \eqref{polar_form_t} be sufficiently smooth, e.g., $A(\cdot,t) \,, \vphi(\cdot,t) \in H^2(\RN)$ for all $t \in [0,T]$. Then with the diagonal matrix 
    \begin{equation*}
        \Cb'(t):= \operatorname{diag}(2c_1'(t)\,,\dots\,,2c_N'(t)) : [0,T] \rightarrow  \R^{N\times N} \,,
    \end{equation*}
it follows that
    \begin{equation}\label{eq_TIE_generalized}
		\ddt I(\xv,t) =  \nabla_{\xv} \cdot \kl{ \Cb'(t) \, I(\xv,t) \, \nabla_{\xv} \vphi(\xv,t)} \,,	
	\end{equation}
and, if $A(\xv,t)$ and $\vphi(\xv,t)$ are also differentiable wrt.\ $t$ (and if $\sqrt{I(\xv,t)} \neq 0$), then
    \begin{equation}\label{eq_TPE_generalized}
		2 \ddt \vphi(\xv,t) 
		=  \norm{\Cb'(t)^\frac{1}{2} \nabla_{\xv} \vphi(\xv,t)}_{\RN}^2 -  \frac{1}{\sqrt{I(\xv,t)}} \nabla_{\xv} \cdot \kl{ \Cb'(t)^\frac{1}{2} \nabla_{\xv} \sqrt{I(\xv,t)}} \,.
	\end{equation}
\end{theorem}
\begin{proof}
First, note that by combining \eqref{eq_ddt_I_poly_diff_L} and \eqref{eq_LSc_defocus} we obtain that
    \begin{equation*}
	\begin{split}
	    \ddt I(\xv,t) 
        &=  2 A(\xv,t) \sum\limits_{k=1}^N  c_k'(t)\kl{2\ddxk \vphi(\xv,t) \ddxk A (\xv,t) + A(\xv,t) \dtdxkt \vphi(\xv,t) } 
        \\
        &
        = \sum\limits_{k=1}^N  2 c_k'(t)\kl{\ddxk \vphi(\xv,t) \kl{ \ddxk A^2 (\xv,t)} + A(\xv,t)^2 \dtdxkt \vphi(\xv,t) } \,,
	\end{split}
	\end{equation*}
which together with $A(\xv,t)^2 = I(\xv,t)$ and the definition of the diagonal matrix $\Cb'(t)$ now yields \eqref{eq_TIE_generalized}. Similarly, combining \eqref{eq_ddt_phi_poly_diff_S} and \eqref{eq_LSc_defocus}, we obtain that
    \begin{equation*}
	\begin{split}
	    &2 I(\xv,t)\ddt \vphi(\xv,t) 
        =  2 A(\xv,t) \sum\limits_{k=1}^N  c_k'(t) \kl{ \kl{\ddxk \vphi(\xv,t)}^2 A(\xv,t) -  \dtdxkt A(\xv,t) } 
        \\
        &\qquad 
        = A(\xv,t)^2 \sum\limits_{k=1}^N  2 c_k'(t) \kl{\ddxk \vphi(\xv,t)}^2
        -
        A(\xv,t) \sum\limits_{k=1}^N  2 c_k'(t) \dtdxkt A(\xv,t)  \,,
	\end{split}
	\end{equation*}
which, together with $A(\xv,t) = \sqrt{I(\xv,t)}$, can be rewritten as
    \begin{equation*}
	\begin{split}
	    &2 \ddt \vphi(\xv,t) 
        = \sum\limits_{k=1}^N  2 c_k'(t) \kl{\ddxk \vphi(\xv,t)}^2
        -
        \frac{1}{\sqrt{I(\xv,t)}} \sum\limits_{k=1}^N  2 c_k'(t) \dtdxkt \sqrt{I(\xv,t)}  \,.
	\end{split}
	\end{equation*}
Together with the definition of $\Cb'(t)$ this yields \eqref{eq_TIE_generalized}, completing the proof.
\end{proof}

\begin{remark}
Note that if $c_k(t) = ct/2$ for all $k$, then \eqref{eq_TIE_generalized} and \eqref{eq_TPE_generalized} simplify to  
    \begin{equation}\label{eq_TIE_classic}
		\frac{1}{c}\ddt I(\xv,t) =  \nabla_{\xv} \cdot \kl{ I(\xv,t) \, \nabla_{\xv} \vphi(\xv,t)} \,,
	\end{equation}
and 
    \begin{equation}\label{eq_TPE_classic}
		2 \ddt \vphi(\xv,t) 
		=  c \norm{\nabla_{\xv} \vphi(\xv,t)}_{\RN}^2 -  \frac{\sqrt{c}}{\sqrt{I(\xv,t)}} \Delta_{\xv} \sqrt{I(\xv,t)} \,,
	\end{equation}
respectively. The second equation, i.e., the ODE \eqref{eq_TPE_classic}, is known as the transport of phase equation (TPE), and can be used to compute the evolution of the phase $\vphi(\xv,t)$ over time $t$ from the corresponding intensity $I(\xv,t)$, given the initial phase $\vphi(\xv,0)$ and suitable boundary conditions in space. The initial phase $\vphi(\xv,0)$, or in fact any $\vphi(\xv,t)$, can in turn be computed from the differential measurements $\ddt I(\xv,t)$ using \eqref{eq_TIE_classic}. This PDE is a variant of the transport of intensity equation (TIE) known from physics, i.e.,
    \begin{equation}\label{TIE_physics}
        - \frac {2\pi }{\lambda } \frac {\partial }{\partial z}\mathcal{I}(x,y,z)= \nabla_{x,y} \cdot \kl{ \mathcal{I}(x,y,z)\nabla _{x,y}\Phi(x,y,z) } \,,
    \end{equation}
which connects the intensity $\mathcal{I}(x,y,z)$ and the phase $\Phi(x,y,z)$ of a 3D wave (field) with wavelength $\lambda$ propagating in $z$-direction. Note that since $\psi(\xiv,t) = 2\pi^2 c\abs{ \xiv}^2$ can physically be interpreted as a defocus, and thus the functions $I(\xv,t)$ as successively defocused intensity measurements, the time derivative $\ddt I(\xv,t)$ in the generalized TIE \eqref{eq_TIE_generalized} and the $z$-derivative $\frac {\partial }{\partial z}\mathcal{I}(x,y,z)$ in the classic TIE \eqref{eq_TIE_classic} physically coincide here.
\end{remark}

For a fixed time $t$, the (generalized) TIE \eqref{eq_TIE_generalized} can be understood as an elliptic PDE for $\vphi(\xv,t)$. E.g., if $\Omega \subset \R^2$ is a bounded Lipschitz domain, then one may consider
    \begin{equation}\label{eq_TIE_elliptic}
    \begin{split}
        \nabla_{\xv} \cdot \kl{ \Cb'(t) \, I(\xv,t) \, \nabla_{\xv} \vphi(\xv,t)} &= \ddt I(\xv,t)\,, 
        \qquad \forall \, \xv \in \Omega \,,
        \\
        \vphi(\cdot,t)\vert_{\partial \Omega} &= g_D \,,
        \qquad \quad \qquad \text{on} \,\, \partial \Omega \,,
    \end{split}
	\end{equation}
for some known Dirichlet boundary values $g_D(\xv,t)$. If $\Cb'(t) I(\cdot,t) \in L^\infty(\Omega,\R^{N\times N})$ is positive definite, $\ddt I(\xv,t) \in H^{-1}(\Omega)$, and $g_D(\cdot,t) \in H^{\frac{1}{2}}(\partial \Omega)$, then the Lax-Milgram Lemma guarantees the existence of a unique weak solution $\vphi(\cdot,t) \in H^1(\Omega)$ of \eqref{eq_TIE_elliptic}. Using standard regularity results, one may even obtain $\vphi(\cdot,t) \in H^2(\Omega)$, and in certain cases, $\vphi(\cdot,t)$ can be uniquely extended to a solution $\vphi(\cdot,t) \in H^2(\R^2)$ of \eqref{eq_TIE_generalized}. This approach has been successfully used for phase retrieval in practice \cite{Zuo_2020_TIE_Tutorial}, where \eqref{eq_TIE_elliptic} is for example solved via a finite-element approach. However, the choice of the boundary data $g_D$, which in the absence of more accurate information is typically (but inaccurately) set to $g_D(\xv,t) = 0$, remains a delicate topic; see \cite{Kirisits_Raik_Scherzer_Strohmenger_Yan_2025,Zuo_2020_TIE_Tutorial} for a corresponding~discussion.

\begin{example}
Consider the 1D case $\xv = x$ with $\psi(x,t) = \psi_0(x) + c(t)(2\pi)^2 \xi^2$ for some coefficient $c(t) \in C^1([0,T])$. Then the generalized TIE \eqref{eq_TIE_generalized} takes the simple form
    \begin{equation}\label{helper_TIE_1D}
		\ddt I(x,t) =  \ddx\kl{ 2c'(t) I(x,t)  \ddx \vphi(x,t)} \,.
	\end{equation}
Now if $\Omega = [a,b]$, $c'(t)I(x,t) \neq 0$, and all of the integrals below exist, the solution of \eqref{helper_TIE_1D} with the boundary conditions $\vphi(a,t) = g_1(t)$ and $\vphi(b,t) = g_2(t)$ is given by
    \begin{equation*}
        \vphi(x,t) = g_1(t) 
        +
        \frac{1}{2c'(t) } \kl{
        \int_a^x \frac{1}{I(r,t)} \kl{\int_a^r \ddt I(s,t) \,ds} \, dr
        + C(t) \int_a^x \frac{1}{I(r,t)} \,dr } \,,
    \end{equation*}
where
    \begin{equation*}
        C(t) =
        \kl{\int_a^b \frac{1}{2c'(t) I(r,t)}\,dr }^{-1} \kl{
        g_2(t) - g_1(t) - \int_a^b \frac{1}{2c'(t) I(r,t)} \kl{\int_a^r \ddt I(s,t) \,ds } \, dr } \,.
    \end{equation*}
Alternatively, one can solve \eqref{helper_TIE_1D} by using the property \eqref{Fourier_diff}, which (formally) yields 
    \begin{equation}\label{TIE_FM}
        \vphi(x,t) = \FI \kl{ \frac{1}{2 \pi i \xi} \F\kl{ \frac{1}{2c'(t) I(\cdot,t)} \FI \kl{ \frac{1}{2 \pi i \xi}\F \kl{ \ddt I(\cdot,t) } } } } \,.
    \end{equation}
This approach will be useful when considering Fourier-type wavefront sensors below.
\end{example}

% % % % % % % % % % % % % % % % % % %
% Subsection - General Modulations  %
% % % % % % % % % % % % % % % % % % %
\subsection{General Modulations}

Next, we consider general modulation functions $\psi(\xiv,t)$ not necessarily polynomial. Note first that if $\psi(\xiv,t)$ is sufficiently smooth, Taylor approximation wrt.\ $\xiv$ around $0$ yields
    \begin{equation*}
        \psi(\xiv,t) 
        = 
        \sum_{\alpha} d_\alpha(t) \xiv^\alpha \,,
        \qquad
        \text{where}
        \qquad
        d_\alpha(t) := \frac{1}{\alpha!} D_{\xiv}^\alpha \psi(\xiv,t) \big\vert_{\xiv = 0} \,,
    \end{equation*}
and thus there holds
    \begin{equation*}
        \psi(\xiv,t) 
        \approx 
        \psi_M(\xiv,t) :=
        \sum_{\abs{\alpha} \leq M} d_\alpha(t) \xiv^\alpha \,.
    \end{equation*}
Since $\psi_M(\xiv,t)$ is a polynomial of the form \eqref{psi_poly}, our previous results are applicable, leading to approximate PDEs for the unknown phase $\vphi(\xv,t)$. While this approach is certainly possible, its main drawback is that a very high order $M$ may be required to obtain a reasonably accurate approximation of the original modulation $\psi(\xiv,t)$. Hence, the resulting PDEs will also be of high order, and thus difficult to solve in practice. 

Hence, we now pursue an alternative approach, making use of the following

\begin{assumption}\label{ass_G}
The modulation function $\psi(\xiv,t)$ is such that the linear operator
    \begin{equation}\label{def_G}
    \begin{split}
        \G_\psi : X \subset \LtRN \to \LtRN \,,
        \qquad v
        \mapsto
        \G_\psi(v) := \FI\kl{ \psi'(t) \F(v)  }
    \end{split}
    \end{equation}
is well-defined and bounded for all $t \in [0,T]$. Here, $X$ is a closed subspace of $\LtRN$.
\end{assumption}

\begin{remark}
Note that $\G_\Psi$ can be understood as a pseudo-differential operator with symbol $p(\xv,\xiv) = \ddt \psi (\xiv, t)$. In particular, for $\psi(\xiv,t)$ as in \eqref{psi_poly}, $\G_\Psi$ is a differential operator of order at most $M$, and thus Assumption~\ref{ass_G} holds, e.g., with $X = H^M(\RN)$.
\end{remark}

Next, we use the operator $\G_\psi$ defined above to write
    \begin{equation*}
         \FI(e^{i\psi(t)} u i \psi'(t) )(\xv)
         =
         i \G_\psi \kl{\FI(e^{i\psi(t)} u  )}(\xv)
         =
         i \G_\psi \kl{A(\cdot,t)e^{i\vphi(\cdot,t)}}(\xv)
         \,.
    \end{equation*}
where we have used the polar form \eqref{polar_form_t}. With this we obtain the following result:

\begin{theorem}\label{thm_helper_general}
Let $I(\xv,t)$ be defined as in \eqref{def_Ixt}, $A(\xv,t)$ and $\vphi(\xv,t)$ as in \eqref{polar_form_t}, let Assumption~\ref{ass_main} and \ref{ass_G} hold, and let $\FI(e^{i\psi(t)} u  ) \in X$ for all $t \in [0,T]$. Then $I(\xv,t)$ is continuously differentiable wrt.\ $t$, i.e., $I(\xv,t) \in C^1([0,T],L^1(\RN))$, and there holds
	\begin{equation}\label{eq_ddt_I_general_G}
		\ddt I(\xv,t) = 2 \Re\kl{i A(\xv,t)e^{-i\vphi(\xv,t)} \G_\psi\kl{A(\cdot,t)e^{i\vphi(\cdot,t)}}(\xv)} \,.
	\end{equation}
Furthermore, if $A(\xv,t)$ and $\vphi(\xv,t)$ are also differentiable wrt.\ $t$, then
	\begin{equation}\label{eq_ddt_phi_general_G}
		2 I(\xv,t) \ddt \vphi(\xv,t) = \Im\kl{i A(\xv,t)e^{-i\vphi(\xv,t)} \G_\psi\kl{A(\cdot,t)e^{i\vphi(\cdot,t)}}(\xv) } \,.
	\end{equation}
\end{theorem}
\begin{proof}
This follows analogously as in the proof of Theorem~\ref{thm_ddt_I_general}.
\end{proof}

Before proceeding further, we first consider a useful special case in the following

\begin{example}\label{ex_fracLapl_01}
Let $\psi(\xiv,t) := t \abs{2\pi \xiv}^{2s}$. Then $\psi'(t) = \abs{2\pi \xiv}^{2s}$ and thus we have
    \begin{equation*}
        \G_\psi(v) = \FI\kl{  \abs{2\pi \xiv}^{2s} \F(v)  }
        \overset{\eqref{def_fracLapl}}{=}
        (-\Delta)^{s} v \,.
    \end{equation*}
Hence, if $\FI(e^{i\psi(t)} u  ) \in D\kl{(-\Delta)^{s}} \subset H^{2s}(\RN)$, then Theorem~\ref{thm_helper_general} yields
    \begin{equation*}
		\ddt I(\xv,t) = 2 \Re\kl{i A(\xv,t)e^{-i\vphi(\xv,t)} (-\Delta)^{s} \kl{A(\cdot,t)e^{i\vphi(\cdot,t)}}(\xv)} \,.
	\end{equation*}
and
    \begin{equation*}
		2 I(\xv,t) \ddt \vphi(\xv,t) = \Im\kl{i A(\xv,t)e^{-i\vphi(\xv,t)} (-\Delta)^{s}\kl{A(\cdot,t)e^{i\vphi(\cdot,t)}}(\xv) } \,.
	\end{equation*}
For $s \in \N$, this can be seen to coincide with our results from the polynomial case. However, for $s \in \R \setminus \N$, $(-\Delta)^s$ is a non-local operator, and does not satisfy the classic chain or product rules, which were the basis for our further analysis based on \eqref{eq_Pcat_chain}.
\end{example}

In order to continue with our analysis, we make the following assumption on $\G_\psi$:

\begin{assumption}\label{ass_general}
For some $Z_1,Z_2 \subseteq \LtRN$, there exist real-valued operators 
    \begin{equation*}
        \L_\psi : Z_1 \times Z_2 \to  \LtRN \,,
        \qquad \text{and} \qquad
        \S_\psi : Z_1 \times Z_2 \to  \LtRN \,,
    \end{equation*}
such that for all $A \in Z_1$ and $\vphi \in Z_2$ there holds
    \begin{equation}\label{eq_ass_L_S}
    \begin{split}
		&i \G_\psi \kl{Ae^{i\vphi}}(\xv)
        = e^{i \vphi(\xv)} \kl{ \L_\psi(A,\vphi)(\xv) + i \S_\psi(A,\vphi)(\xv) } \,.
    \end{split}
	\end{equation}
\end{assumption}

With this, we now obtain an analogon to Theorem~\ref{thm_main_polynomials} for general modulations.

\begin{theorem}\label{thm_main_general}
Let Assumption~\ref{ass_general} and the assumptions of Theorem~\ref{thm_helper_general} hold. Then if $A(\cdot,t) \in Z_1$ and $\vphi( \cdot,t) \in Z_2$ for all $t \in [0,T]$, there holds
    \begin{equation}\label{eq_ddt_I_general_G_LS}
		\ddt I(\xv,t) =  2 A(\xv,t) \L_\psi(A(\cdot,t),\vphi(\cdot,t))(\xv) \,,
	\end{equation}
and furthermore, if $A(\xv,t)$ and $\vphi(\xv,t)$ are also differentiable wrt.\ $t$, then
    \begin{equation}\label{eq_ddt_phi_general_G_LS}
		2 I(\xv,t)\ddt \vphi(\xv,t)  =   2 A(\xv,t) \S_\psi(A(\cdot,t),\vphi(\cdot,t))(\xv) \,.
	\end{equation}
\end{theorem}
\begin{proof}
This directly follows by combining \eqref{eq_ass_L_S} with \eqref{eq_ddt_I_general_G} and \eqref{eq_ddt_phi_general_G}, respectively. 
\end{proof}

Assumption~\ref{ass_general} clearly generalizes \eqref{eq_Pcat_chain} from the polynomial case. Note that in contrast to \eqref{eq_Pcat_chain}, $\L_\psi$ and $\S_\psi$ do not have to be differential operators, nor do they have to be linear. Furthermore, if $(\G_\psi v)(\xv)$ is either purely real- or complex-valued for real-valued $v(\xv) \in X$, then condition \eqref{eq_ass_L_S} is always satisfied, as we now show in 

\begin{proposition}\label{prop_G_split}
Assume that $\G_\psi$ defined in \eqref{def_G} is such that $\G_\psi v$ is real-valued for all real-valued $v \in X$. Furthermore, assume that $Z_1,Z_2 \subseteq \LtRN$ are such that 
    \begin{equation}\label{cond_Z1_Z2}
        A \cos(\vphi) \in X 
        \,, \quad \text{and} \quad
        A \sin(\vphi) \in X \,,
        \qquad 
        \forall 
        A \in Z_1 \,, \vphi \in Z_2 \,.
    \end{equation}
Then $\G_\psi$ satisfies Assumption~\ref{ass_G}, and in particular condition \eqref{eq_ass_L_S}, with
    \begin{equation*}
    \begin{split}
        \L_\psi(A,\vphi)
        &=
        \sin\kl{\vphi}
        \G_\psi \kl{A \cos\kl{\vphi}}
        -
        \cos\kl{\vphi}\G_\psi \kl{A \sin\kl{\vphi}} \,,
        \\
        \S_\psi(A,\vphi) 
        &=
        \cos\kl{\vphi} 
        \G_\psi  \kl{A \cos\kl{\vphi}}
        +
        \sin\kl{\vphi}\G_\psi  \kl{A \sin\kl{\vphi}} \,.
    \end{split}
    \end{equation*}
Furthermore, if instead $i \G_\psi v$ is real-valued for all real-valued $v\in X$ and \eqref{cond_Z1_Z2} holds, then $\G_\psi$ satisfies condition \eqref{eq_ass_L_S} with swapped right-hand sides in the above equations.
\end{proposition}
\begin{proof}
Let $A \in Z_1$, $\vphi \in Z_2$ be arbitrary but fixed. Due to the Euler formula, we have
    \begin{equation*}
    \begin{split}
		&i\G_\psi \kl{A e^{i\vphi}}(\xv) 
        = 
        i \G_\psi \kl{A \cos\kl{\vphi}}(\xv)
        -
        \G_\psi \kl{A \sin\kl{\vphi}}(\xv)
        \,,
    \end{split}
	\end{equation*}
where we have used the linearity of $\G_\psi$. Together with
    \begin{equation*}
        1 = e^{i \vphi} \kl{\cos\kl{\vphi} - i \sin\kl{\vphi}} \,,
    \end{equation*}
we thus find that
    \begin{equation*}
    \begin{split}
		i\G_\psi \kl{Ae^{i\vphi}}(\xv) 
        =& 
        e^{i \vphi} \kl{\cos\kl{\vphi} - i \sin\kl{\vphi}}
        i\G_\psi \kl{A \cos\kl{\vphi}}(\xv) 
        \\
        & 
        -
        e^{i \vphi} \kl{\cos\kl{\vphi} - i \sin\kl{\vphi}}\G_\psi \kl{A \sin\kl{\vphi}}(\xv)
        \,.
    \end{split}
	\end{equation*}
Hence, if $\G_\psi v$ is real-valued for real-valued $v \in X$, it follows that \eqref{eq_ass_L_S} is satisfied with
    \begin{equation*}
    \begin{split}
        \L_\psi(A,\vphi)
        &:=
        \sin\kl{\vphi}
        \G_\psi \kl{A \cos\kl{\vphi}}
        -
        \cos\kl{\vphi}\G_\psi \kl{A \sin\kl{\vphi}} \,,
        \\
        \S_\psi(A,\vphi) 
        &:=
        \cos\kl{\vphi} 
        \G_\psi \kl{A \cos\kl{\vphi}}
        +
        \sin\kl{\vphi}\G_\psi \kl{A \sin\kl{\vphi}} \,,
    \end{split}
    \end{equation*}
which yields the first part of the assertion. Now if $\G_\psi v)$ is purely complex-valued for real-valued $v \in X$, then $i \G_\psi v$ is real-valued, and it follows that \eqref{eq_ass_L_S} is satisfied with
    \begin{equation*}
    \begin{split}
        \L_\psi(A,\vphi)
        &:=
        \cos\kl{\vphi}
        i\G_\psi \kl{A \cos\kl{\vphi}}
        +
        \sin\kl{\vphi}i\G_\psi \kl{A \sin\kl{\vphi}} \,,
        \\
        \S_\psi(A,\vphi) 
        &:=
        \sin\kl{\vphi}
        i\G_\psi \kl{A \cos\kl{\vphi}}
        -
        \cos\kl{\vphi}i\G_\psi \kl{A \sin\kl{\vphi}} \,,
    \end{split}
    \end{equation*}
which yields the second part of the assertion and thus completes the proof. 
\end{proof}

Note that from the definition \eqref{def_G} of $\G_\psi$, it follows that for all real-valued $v \in X$,
    \begin{equation}\label{obs_real}
    \begin{split}
        \G_\psi v \,\, \text{is real-valued} 
        &\quad \Longleftrightarrow \quad
        \psi(-\xiv,t) = \psi(\xiv,t) \,, \quad \forall \, \xiv \in \RN \,, t \in [0,T] \,,
        \\
        i\G_\psi v \,\, \text{is real-valued} 
        &\quad \Longleftrightarrow \quad
        \psi(-\xiv,t) = -\psi(\xiv,t) \,, \quad \forall \, \xiv \in \RN \,, t \in [0,T] \,,
    \end{split}
    \end{equation}
i.e., if $\psi(\cdot,t)$ is odd or even, respectively. Hence, combining the above results, we obtain

\begin{theorem}\label{thm_main_wave}
Let the assumptions of Theorem~\ref{thm_helper_general} hold, let $Z_1,Z_2 \subseteq \LtRN$ be such that \eqref{cond_Z1_Z2} holds, and let $A(\cdot,t) \in Z_1$ and $\vphi( \cdot,t) \in Z_2$ for all $t \in [0,T]$. Then if $\psi(\xiv,t)$ is even wrt.\ $\xiv$, i.e., if $\psi(-\xiv,t) = \psi(\xiv,t)$ for all $\xiv \in \RN$ and $t \in [0,T]$, there~holds
    \begin{equation}\label{eq_main_TIE}
    \begin{split}
		\ddt I(\xv,t) =&  2 A(\xv,t)
        \sin\kl{\vphi(\xv,t)}
        \G_\psi \kl{A(\cdot,t) \cos\kl{\vphi(\cdot,t)}}(\xv)
        \\
        &-
        2 A(\xv,t) \cos\kl{\vphi(\xv,t)}\G_\psi \kl{A(\cdot,t) \sin\kl{\vphi(\cdot,t)}}(\xv) \,,
    \end{split}
	\end{equation}
and furthermore, if $A(\xv,t)$ and $\vphi(\xv,t)$ are also differentiable wrt.\ $t$, then
    \begin{equation}\label{eq_main_TPE}
    \begin{split}
		2 I(\xv,t)\ddt \vphi(\xv,t)  =&   2 A(\xv,t) 
        \cos\kl{\vphi(\xv,t)} 
        \G_\psi \kl{A(\cdot,t) \cos\kl{\vphi(\cdot,t)}}(\xv)
        \\
        & + 2 A(\xv,t) \sin\kl{\vphi(\xv,t)}\G_\psi \kl{A(\cdot,t) \sin\kl{\vphi(\cdot,t)}}(\xv)\,.
    \end{split}
	\end{equation}   
If instead $\psi(\xiv,t)$ is odd wrt.\ $\xiv$, i.e., if $\psi(-\xiv,t) = -\psi(\xiv,t)$ for all $\xiv \in \RN$ and $t \in [0,T]$, then the above two equations hold with swapped right-hand sides.
\end{theorem}
\begin{proof}
This follows by combining Proposition~\ref{prop_G_split} and \eqref{obs_real} with Theorem~\ref{thm_main_general}.    
\end{proof}

The following remark on approximate equations will become useful in Section~\ref{sect_WFS}.

\begin{remark}
In the setting of the above theorem, consider the case that $\vphi(\xv,t)$ is sufficiently small such that $\sin(\vphi(\xv,t)) \approx \vphi(\xv,t)$ and $\cos(\vphi(\xv,t)) \approx 1$. (Note that here ``small'' means small wrt.\ to the chosen norm on $X$.) Then formally,~\eqref{eq_main_TIE}~reads 
    \begin{equation*}
        \ddt I(\xv,t) \approx   2 A(\xv,t)
        \vphi(\xv,t)
        \G_\psi \kl{A(\cdot,t)}(\xv)
        -
        2 A(\xv,t) \G_\psi \kl{A(\cdot,t) \vphi(\cdot,t)}(\xv)  \,,
    \end{equation*}
which we can rewrite as
    \begin{equation*}
        \ddt I(\xv,t) \approx   2 \kl{ \rule{0pt}{2.5ex} \G_\psi \kl{A(\cdot,t)}(\xv) \, \mathcal{I} -  A(\xv,t) \G_\psi  }\kl{ A(\xv,t)
        \vphi(\xv,t) } \,.
    \end{equation*}
Assuming that $ \kl{ \G_\psi \kl{A(\cdot,t)}(\xv) \, \mathcal{I} -  A(\xv,t) \G_\psi  }$ is invertible and $A(\xv,t) \neq 0$, this implies
    \begin{equation*}
        \boxed{\vphi(\xv,t) \approx  \frac{1}{2 A(\xv,t)}\kl{  \rule{0pt}{2.5ex} \G_\psi \kl{A(\cdot,t)}(\xv) \, \mathcal{I} -  A(\xv,t) \G_\psi  }^{-1} \kl{ \ddt I(\xv,t) }  \,,}
    \end{equation*}
which is an explicit reconstruction formula for the phase $\vphi(\xv,t)$. On the other hand, if $\psi(\cdot,t)$ is odd, such that \eqref{eq_main_TIE} holds with the swapped right-hand side of \eqref{eq_main_TPE}, i.e., 
    \begin{equation*}
    \begin{split}
		\ddt I(\xv,t) =&   2 A(\xv,t) 
        \cos\kl{\vphi(\xv,t)} 
        i\G_\psi \kl{A(\cdot,t) \cos\kl{\vphi(\cdot,t)}}(\xv)
        \\
        & + 2 A(\xv,t) \sin\kl{\vphi(\xv,t)}i\G_\psi \kl{A(\cdot,t) \sin\kl{\vphi(\cdot,t)}}(\xv) \,,
    \end{split}
	\end{equation*}
then with the approximation $\sin(\vphi(\xv,t)) \approx \vphi(\xv,t)$ and $\cos(\vphi(\xv,t)) \approx 1$ we obtain
    \begin{equation*}
		\ddt I(\xv,t) =   2 A(\xv,t)  
        i\G_\psi \kl{A(\cdot,t) }(\xv)
         + 2 A(\xv,t) \vphi(\xv,t)i\G_\psi \kl{A(\cdot,t) \vphi(\cdot,t)}(\xv) \,,
	\end{equation*}
which, again assuming $A(\xv,t) \neq 0$, can be reformulated into the nonlinear equation
    \begin{equation*}
		\vphi(\xv,t)i\G_\psi \kl{A(\cdot,t) \vphi(\cdot,t)}(\xv) 
        =
        \frac{1}{2 A(\xv,t)} \kl{ \ddt I(\xv,t) - 2 A(\xv,t)  
        \G_\psi \kl{A(\cdot,t) }(\xv) } \,.
	\end{equation*}
The generalized transport of phase equation \eqref{eq_main_TPE} can be simplified analogously.
\end{remark}

Finally, note that not every $\psi(\xiv,t)$ leads to a useful equation for $\vphi(\xv,t)$, as we see~in 

\begin{example}[Not every $\psi$ is useful]\label{ex_not_useful}
Consider the following two special cases:
\begin{itemize}
    \item $\psi(\xiv,t) = c t$. In this case we have $\psi'(t) = c$, $\G_\psi(v) = c v$, and thus
        \begin{equation*}
            i A(\xv,t)e^{-i\vphi(\xv,t)} \G_\psi\kl{A(\cdot,t)e^{i\vphi(\cdot,t)}}(\xv) 
            =
            i c A(\xv,t)^2
            =
            i c I(\xv,t) \,.
        \end{equation*}
    Hence, \eqref{eq_ddt_I_general_G} and \eqref{eq_ddt_phi_general_G} simplify to
        \begin{equation*}
		      \ddt I(\xv,t) = 0 \,,
            \qquad
            \text{and}
            \qquad
            2 I(\xv,t) \ddt \vphi(\xv,t) = c I(\xv,t) \,,
	    \end{equation*}
    which both are not useful for determining the phase $\vphi(\xv,t)$. This is not very surprising, since this choice of $\psi(\xiv,t)$ does not induce intensity changes in $I(\xv,t)$.
    \item $\psi(\xiv,t) = 2\pi (\xiv_1 + \dots + \xiv_N) t$. In this case, we have $\psi'(t) = 2\pi (\xiv_1 + \dots + \xiv_N)$,
        \begin{equation*}
            i\G_\psi(v)(\xv) = \sum_{k=1}^N \FI\kl{ i 2\pi \xi_k \F(v)  }
            \overset{\eqref{Fourier_diff}}{=}
            \sum_{k=1}^N \frac{\partial}{\partial_{x_k}} v(\xv)\,, 
        \end{equation*}
        and thus
        \begin{equation*}
        \begin{split}
            &i A(\xv,t)e^{-i\vphi(\xv,t)} \G_\psi\kl{A(\cdot,t)e^{i\vphi(\cdot,t)}}(\xv) 
            =
            A(\xv,t)e^{-i\vphi(\xv,t)} \sum_{k=1}^N \frac{\partial}{\partial_{x_k}}  \kl{A(\cdot,t)e^{i\vphi(\cdot,t)}}(\xv)
            \\
            &\qquad
            = A(\xv,t) \sum_{k=1}^N \frac{\partial}{\partial_{x_k}} 
            A(\xv,t)
            +
            i A(\xv,t)^2 \sum_{k=1}^N \frac{\partial}{\partial_{x_k}} \vphi(\xv,t)\,.
        \end{split}
        \end{equation*} 
    Hence, \eqref{eq_ddt_I_general_G} and \eqref{eq_ddt_phi_general_G} simplify to
        \begin{equation}\label{eq_helper_linear_I}
		      \ddt I(\xv,t) = 2 A(\xv,t) \sum_{k=1}^N \frac{\partial}{\partial_{x_k}} 
            A(\xv,t) \,,
        \end{equation}
    and
        \begin{equation}\label{eq_helper_linear_phi}
		      2 I(\xv,t) \ddt \vphi(\xv,t) = A(\xv,t)^2 \sum_{k=1}^N \frac{\partial}{\partial_{x_k}} \vphi(\xv,t) \,,
	    \end{equation}
    respectively. Note that as in the first case, \eqref{eq_helper_linear_I} does not include $\vphi(\xv,t)$, and is thus not useful for phase reconstruction. Again, this is not very surprising, since in this case $\psi(\xiv,t)$ merely induces a linear shift in the intensity $I(\xv,t)$.
\end{itemize}
\end{example}

% % % % % % % % % % % % % % % % % % % % % % % % % %
% Subsection - Multiple Differential Measurements %
% % % % % % % % % % % % % % % % % % % % % % % % % %
\subsection{Complex Modulations and Multiple Measurements}\label{sect_multiplemeasurements}

Up to now, we have only considered real-valued modulation functions $\psi(\xiv,t)$. This was motivated by the fact that pure phase modulations $e^{i \psi(\xiv,t)}$ can be implemented physically using suitable optical elements or SLMs. However, one can also consider complex-valued modulation functions $\psi(\xiv,t) + i \omega(\xiv,t)$, for which
    \begin{equation*}
        e^{i \kl{\psi(\xiv,t) + i \omega(\xiv,t)}}
        =
        e^{i \psi(\xiv,t) } e^{-\omega(\xiv,t)}
        \qquad
        \text{in}
        \qquad
        I(\xv,t) := \abs{\FI (e^{i \kl{\psi(\xiv,t) + i \omega(\xiv,t)}} u)(\xv)}^2
    \end{equation*}
corresponds to both amplitude and phase modulations. Physically, amplitude modulations can be realized by e.g.\ using absorption masks or amplitude modulators.

Our above analysis can be easily generalized to this case: For example, using
    \begin{equation*}
    \begin{split}
        \G_{\psi,\omega} : X \subset \LtRN \to \LtRN \,,
        \qquad v
        \mapsto
        \G_{\psi,\omega}(v) := \FI\kl{ \kl{ \psi'(t) + i \omega'(t) }\F(v)  }
    \end{split}
    \end{equation*}
we obtain analogously as in Theorem~\ref{thm_helper_general} that
    \begin{equation}\label{eq_ddt_I_general_G_comp}
		\ddt I(\xv,t) = 2 \Re\kl{i A(\xv,t)e^{-i\vphi(\xv,t)} \G_{\psi,\omega}\kl{A(\cdot,t)e^{i\vphi(\cdot,t)}}(\xv)} \,,
	\end{equation}
and 
	\begin{equation}\label{eq_ddt_phi_general_G_comp}
		2 I(\xv,t) \ddt \vphi(\xv,t) = \Im\kl{i A(\xv,t)e^{-i\vphi(\xv,t)} \G_{\psi,\omega}\kl{A(\cdot,t)e^{i\vphi(\cdot,t)}}(\xv) } \,.
	\end{equation} 
Hence, if $\G_{\psi,\omega}$ allows for a decomposition of the form
    \begin{equation*}
        i \G_{\psi,\omega} \kl{Ae^{i\vphi}}(\xv)
        = e^{i \vphi(\xv)} \kl{ \L_{\psi,\omega}(A,\vphi)(\xv) + i \S_{\psi,\omega}(A,\vphi)(\xv) } \,,
    \end{equation*}
for real-valued operators $\L_{\psi,\omega}$ and $\S_{\psi,\omega}$, then the above equations simplify to
    \begin{equation*}
		\ddt I(\xv,t) = 2  A(\xv,t) \L_{\psi,\omega}(A(\cdot,t),\vphi(\cdot,t))(\xv) \,,
	\end{equation*}
and 
	\begin{equation*}
		2 I(\xv,t) \ddt \vphi(\xv,t) = A(\xv,t) \S_{\psi,\omega}(A(\cdot,t),\vphi(\cdot,t))(\xv) \,,
	\end{equation*} 
respectively. Note that in the purely complex case $\psi(\xiv,t) = 0$, we have that
    \begin{equation*}
        \G_{\psi,\omega}(v) = \FI\kl{ i \omega'(t) \F(v)  }
        =
        i \G_\omega(v) \,.
    \end{equation*}
Hence, if $\G_\omega$ allows for a decomposition of the form \eqref{eq_ass_L_S}, then we obtain the equations
    \begin{equation*}
		\ddt I(\xv,t) = -2  A(\xv,t) \S_{\omega}(A(\cdot,t),\vphi(\cdot,t))(\xv) \,,
	\end{equation*}
and 
	\begin{equation*}
		2 I(\xv,t) \ddt \vphi(\xv,t) = A(\xv,t) \L_{\omega}(A(\cdot,t),\vphi(\cdot,t))(\xv) \,,
	\end{equation*} 
which are essentially the same as in Theorem~\ref{thm_main_general} but with the roles of $\L$ and $\S$ switched.

The use of complex modulations can perhaps be best illustrated by the following

\begin{example}
Consider the purely complex modulation 
$\psi(\xiv,t) = i 2\pi (\xiv_1 + \dots + \xiv_N) t$. In this case, we have $\psi'(t) = i 2\pi (\xiv_1 + \dots + \xiv_N)$, and thus there holds
        \begin{equation*}
            \G_\psi(v)(\xv) = \sum_{k=1}^N \FI\kl{ i 2\pi \xi_k \F(v)  }
            \overset{\eqref{Fourier_diff}}{=}
            \sum_{k=1}^N \frac{\partial}{\partial_{x_k}} v(\xv)\,.
        \end{equation*}
Hence, we find as in Example~\ref{ex_not_useful} that \eqref{eq_ddt_I_general_G_comp} and \eqref{eq_ddt_phi_general_G_comp} simplify to
        \begin{equation}\label{eq_helper_linear_I_complex}
		      \ddt I(\xv,t) = - 2 A(\xv,t)^2 \sum_{k=1}^N \frac{\partial}{\partial_{x_k}} \vphi(\xv,t)
            \,,
        \end{equation}
    and
        \begin{equation*}
		      2 I(\xv,t) \ddt \vphi(\xv,t) = A(\xv,t) \sum_{k=1}^N \frac{\partial}{\partial_{x_k}} 
            A(\xv,t) \,,
	    \end{equation*}    
which, as expected, is exactly \eqref{eq_helper_linear_I} and \eqref{eq_helper_linear_phi} with swapped right-hand sides. Note, however, that \eqref{eq_helper_linear_I_complex} now does include $\vphi(\xv,t)$, making it useful for phase reconstruction.
\end{example}

Finally, note that instead of a single set of time-dependent measurements $I(\xv,t)$, one may also consider several time dependent measurements $I_{\psi_k}(\xv,t)$ of the form
    \begin{equation*}
        I_{\psi_k}(\xv,t) := \abs{\FI\kl{e^{i\psi_k(\xv,t)} u}}^2 \,,
        \qquad
        \forall \, k = 1\,,\dots \,, K \,.
    \end{equation*}
These measurements can yield complementary information useful for reconstruction, as illustrated by the following example: Assume that $K = 2$, and consider the modulations
    \begin{equation*}
        \psi_1(\xiv,t) := \psi_0(\xiv) + (2\pi)^2 \xi_1^2 t 
        \qquad \text{and} \qquad
        \psi_2(\xiv,t) := \psi_0(\xiv) + (2\pi)^2 \xi_2^2 t \,. 
    \end{equation*}
which physically can be realized via cylindrical lenses. Using our results of Section~\ref{subsect_defoucs}, in particular \eqref{eq_TIE_generalized}, for each of the corresponding intensities $I_{\psi_k}(\xv,t)$, we find that
    \begin{equation*}
	\begin{split}
	   &\ddt I_{\psi_1}(\xv,t) 
	   = \frac{\partial}{\partial x_1 }\kl{ I_{\psi_1}(\xv,t) \frac{\partial}{\partial x_1} \vphi(\xv,t)} \,,
	   \\
	   & \ddt I_{\psi_2}(\xv,t) 
	   = 
	   \frac{\partial}{\partial x_2 }\kl{ I_{\psi_2}(\xv,t) \frac{\partial}{\partial x_2} \vphi(\xv,t)} 
	   \,.
	\end{split}
	\end{equation*}
These two equations can be used to compute the derivatives $\frac{\partial}{\partial x_1} \vphi(\xv,t)$ and $\frac{\partial}{\partial x_2} \vphi(\xv,t)$, from which in turn $\vphi(\xv,t)$ can be recovered, e.g., by using the efficient CuReD method used for wavefront reconstruction via the Shack-Hartmann wavefront sensor (see below).

% % % % % % % % % % % % % % % % % % % % % % % %
% Section - Application to Fourier-type WFSs  %
% % % % % % % % % % % % % % % % % % % % % % % %
\section{Application to Fourier-type Wavefront Sensors}\label{sect_WFS}

In this section, we discuss the application of our above results to wavefront sensing via Fourier-type WFSs, and derive some new and efficient reconstruction algorithms.

% % % % % % % % % % % % % % % % % % % % % % % % % % % % % % %
% Subsection - Background on Fourier-type Wavefront Sensors %
% % % % % % % % % % % % % % % % % % % % % % % % % % % % % % %
\subsection{Background on Fourier-type Wavefront Sensors}

First, we recall some physical background on Fourier-type WFSs and their mathematical modeling. WFSs are typically used to measure wavefront aberrations in optical systems, and are a key component of adaptive optics (AO) systems used, e.g., in astronomy or medical imaging \cite{Roddier_1999,Pircher_Zawadzki_2017}. They are necessary, since normal cameras/detectors are only able to measure intensities but not phases, which encode the wavefront aberrations. 

The idea behind wavefront sensing is to modify the incoming light in such a way that its modified intensity contains (indirect) information on the original phase, from which it can then be reconstructed mathematically. The necessary modifications are, e.g., achieved using lenses, optical elements, or SLMs. For example, the well-known Shack-Hartmann wavefront sensor (SH-WFS) \cite{Ellerbroek_Vogel_2009,Platt_Shack_2001,Primot_2003} uses a quadratic array of small lenses, also called lenslets, located in the focal plane of a CCD photon detector, to essentially measure the average local gradient of the phase, from which the phase itself can then be efficiently reconstructed numerically; see, e.g.,\cite{Zhariy_Neubauer_Rosensteiner_Ramlau_2011,Rosensteiner_2011_01,Rosensteiner_2011_02}.

% Figure - Image Formation Model 
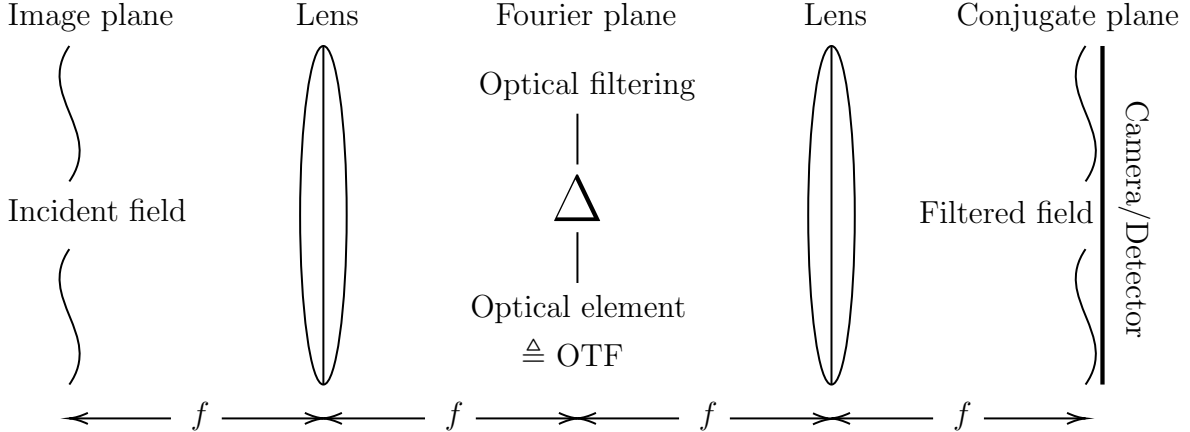
\begin{figure}[ht!]
    \centering
    \begin{tikzpicture}[x=0.75pt,y=0.75pt,yscale=-0.85,xscale=0.85]
    
        % Ellipses representing lenses
        \draw (176.25,150) .. controls (176.25,94.77) and (182.41,50) .. (190,50) .. controls (197.59,50) and (203.75,94.77) .. (203.75,150) .. controls (203.75,205.23) and (197.59,250) .. (190,250) .. controls (182.41,250) and (176.25,205.23) .. (176.25,150) -- cycle;
        \draw (476.25,150) .. controls (476.25,94.77) and (482.41,50) .. (490,50) .. controls (497.59,50) and (503.75,94.77) .. (503.75,150) .. controls (503.75,205.23) and (497.59,250) .. (490,250) .. controls (482.41,250) and (476.25,205.23) .. (476.25,150) -- cycle;
    
        % Straight lines representing Lenses and Fourier Plane
        \draw (340,90) -- (340,120);
        \draw (340,160) -- (340,190);
        \draw (190,50) -- (190,250);
        \draw (490,50) -- (490,250);
        \draw [ultra thick] (650,50) -- (650,250);
        
        % Curved lines representing light waves
        \draw (40,130) .. controls (60.8,100.3) and (19.8,79.3) .. (40,50);
        \draw (40,250) .. controls (60.8,220.3) and (19.8,199.3) .. (40,170);
        \draw (640,130) .. controls (660.8,100.3) and (619.8,79.3) .. (640,50);
        \draw (640,250) .. controls (660.8,220.3) and (619.8,199.3) .. (640,170);

        % Arrows representing focal lengths
        \draw (100,270) -- (42,270);
        \draw [shift={(40,270)}, rotate = 360][color={rgb, 255:red, 0; green, 0; blue, 0}][line width=0.75] (10.93,-3.29) .. controls (6.95,-1.4) and (3.31,-0.3) .. (0,0) .. controls (3.31,0.3) and (6.95,1.4) .. (10.93,3.29);
        \draw (250,270) -- (192,270);
        \draw [shift={(190,270)}, rotate = 360][color={rgb, 255:red, 0; green, 0; blue, 0}][line width=0.75] (10.93,-3.29) .. controls (6.95,-1.4) and (3.31,-0.3) .. (0,0) .. controls (3.31,0.3) and (6.95,1.4) .. (10.93,3.29);
        \draw (400,270) -- (342,270);
        \draw [shift={(340,270)}, rotate = 360][color={rgb, 255:red, 0; green, 0; blue, 0}][line width=0.75] (10.93,-3.29) .. controls (6.95,-1.4) and (3.31,-0.3) .. (0,0) .. controls (3.31,0.3) and (6.95,1.4) .. (10.93,3.29);
        \draw (550,270) -- (492,270);
        \draw [shift={(490,270)}, rotate = 360][color={rgb, 255:red, 0; green, 0; blue, 0}][line width=0.75] (10.93,-3.29) .. controls (6.95,-1.4) and (3.31,-0.3) .. (0,0) .. controls (3.31,0.3) and (6.95,1.4) .. (10.93,3.29);
        \draw (130,270) -- (188,270);
        \draw [shift={(190,270)}, rotate = 180] [color={rgb, 255:red, 0; green, 0; blue, 0}][line width=0.75] (10.93,-3.29) .. controls (6.95,-1.4) and (3.31,-0.3) .. (0,0) .. controls (3.31,0.3) and (6.95,1.4) .. (10.93,3.29);
        \draw (280,270) -- (338,270);
        \draw [shift={(340,270)}, rotate = 180] [color={rgb, 255:red, 0; green, 0; blue, 0}][line width=0.75] (10.93,-3.29) .. controls (6.95,-1.4) and (3.31,-0.3) .. (0,0) .. controls (3.31,0.3) and (6.95,1.4) .. (10.93,3.29) ;
        \draw (430,270) -- (488,270);
        \draw [shift={(490,270.3)}, rotate = 180.29] [color={rgb, 255:red, 0; green, 0; blue, 0}][line width=0.75] (10.93,-3.29) .. controls (6.95,-1.4) and (3.31,-0.3) .. (0,0) .. controls (3.31,0.3) and (6.95,1.4) .. (10.93,3.29);
        \draw (580,270) -- (638,270);
        \draw [shift={(640,270)}, rotate = 180] [color={rgb, 255:red, 0; green, 0; blue, 0}][line width=0.75] (10.93,-3.29) .. controls (6.95,-1.4) and (3.31,-0.3) .. (0,0) .. controls (3.31,0.3) and (6.95,1.4) .. (10.93,3.29);

        % Text description of focal length
        \draw (110,258.4) node [anchor=north west][inner sep=0.75pt]{$f$};
        \draw (260,258.4) node [anchor=north west][inner sep=0.75pt]{$f$};
        \draw (410,258.4) node [anchor=north west][inner sep=0.75pt]{$f$};
        \draw (560,258.4) node [anchor=north west][inner sep=0.75pt]{$f$};
    
        % Text description of Image and Fourier planes   
        \draw (172,23) node [anchor=north west][inner sep=0.75pt][align=left]{Lens};
        \draw (472,23) node [anchor=north west][inner sep=0.75pt][align=left]{Lens};
        \draw (290,23) node [anchor=north west][inner sep=0.75pt][align=left]{Fourier plane};
        \draw (2,23) node [anchor=north west][inner sep=0.75pt][align=left]{Image plane};
        \draw (562,23) node [anchor=north west][inner sep=0.75pt][align=left]{Conjugate plane};
        \draw (280,63.4) node [anchor=north west][inner sep=0.75pt]{Optical filtering};
        \draw (323,125) node [anchor=north west][inner sep=0.75pt]{\Huge $\Delta$};
        \draw (305,220) node [anchor=north west][inner sep=0.75pt]{$\triangleq$ OTF};
        \draw (275,195) node [anchor=north west][inner sep=0.75pt][align=left]{Optical element};
        \draw (2,138.4) node [anchor=north west][inner sep=0.75pt]{Incident field};
        \draw (540,140) node [anchor=north west][inner sep=0.75pt]{Filtered field};
        \draw (680,80) node [anchor=north west][inner sep=0.75pt][rotate=270]{Camera/Detector};

    \end{tikzpicture}
    \caption{Schematic depiction of a Fourier-type WFS measurement system \cite{HuNeuSha_2023,Hubmer_Laidlaw_Ramlau_Sherina_Stadler_2025}.}
    \label{fig_fwfs}
\end{figure}

Here, we consider the large class of Fourier-type WFSs~\cite{Fauv16,Potiron2019}, of which the pyramid WFS (P-WFS) is perhaps the most well-known instance \cite{Raga96,RaDi02}. The physical principle behind Fourier-type WFSs is illustrated in Figure~\ref{fig_fwfs}: Incoming light is focused onto some type of optical element or SLM located in the focal plane, which modulates/filters the field, before it is again focused onto a camera or a CCD detector. This filtered intensity measurement is then used to numerically reconstruct the unknown phase. In the case of the P-WFS, the optical element is a 4-sided pyramidal prism, which essentially splits the incoming light into four beams, creating four differently filtered  (and potentially overlapping) images on the CCD detector. Extensions to $2$-sided (roof WFS), $3$-sided, and $6$-sided prisms, as well as cone shaped (cone WFS) and general $n$-sided pyramidal prisms haven been investigated in the past~\cite{Clare_ao4elt5,ClareCone2020,engler2017,Veri04}. Alternatively, Fourier-type WFSs not using prismatic elements such as the Zernike WFS \cite{Zernike1934} and the 4QPM WFS \cite{Rouan2000} employ (constant) phase shifts on different parts of the focal plane.

\begin{table}[ht!]
\setlength{\tabcolsep}{15pt}
\renewcommand{\arraystretch}{1.5}
\centering
    \begin{tabular}{|c|c|c|c|}
        \hline  
        & 4-sided PWFS
        & x-Roof WFS
        & Cone WFS
        \\
		\hline 
        $\Psi(\xi_1,\xi_2)=$ 
        & $c\kl{\abs{\xi_1}+\abs{\xi_2}}$ 
        & $c\abs{\xi_1}$ 
        & $ \rule{0ex}{3.5ex} c\sqrt{\abs{\xi_1}^2+\abs{\xi_2}^2}$ 
        \\
        \hline 
        & Defocus WFS
        & y-Roof WFS
        & 4QPM WFS 
        \\
		\hline 
        $\Psi(\xi_1,\xi_2)=$ 
        & $c\kl{\abs{\xi_1}^2+\abs{\xi_2}^2}$ 
        & $\rule{0ex}{4.5ex} c\abs{\xi_2}$ 
        &  $\begin{cases} 
            c \,, & \xi_1 \xi_2 < 0 \,, \\ 0 \,, & \text{otherwise} 
            \end{cases}$ \\
        \hline 
        
        & Riesz WFS
        & Zernike WFS
        & Astigmatism WFS
        \\
        \hline
        $\Psi(\xi_1,\xi_2)=$ 
        & $\rule{0ex}{4ex}  c\displaystyle\frac{\abs{\xi_1}+\abs{\xi_2}}{\abs{\xiv}}$
        & $c\chi_{B_{\rho}(0)}(\xi_1,\xi_2)$
        & $c \, \xi_1 \, \xi_2$ \\
        \hline
	\end{tabular}
    \caption{Shape functions $\Psi(\xiv)$, with $\xiv = (\xi_1,\xi_2)$, for different Fourier-type WFSs. Here, $c>0$ and $\rho > 0$  are a-priori fixed constants. For an illustration, see Figure~\ref{fig_otfs}.}
    \label{table_psi}
\end{table}

\begin{figure}[ht!]
    \centering
    \includegraphics[width=\textwidth]{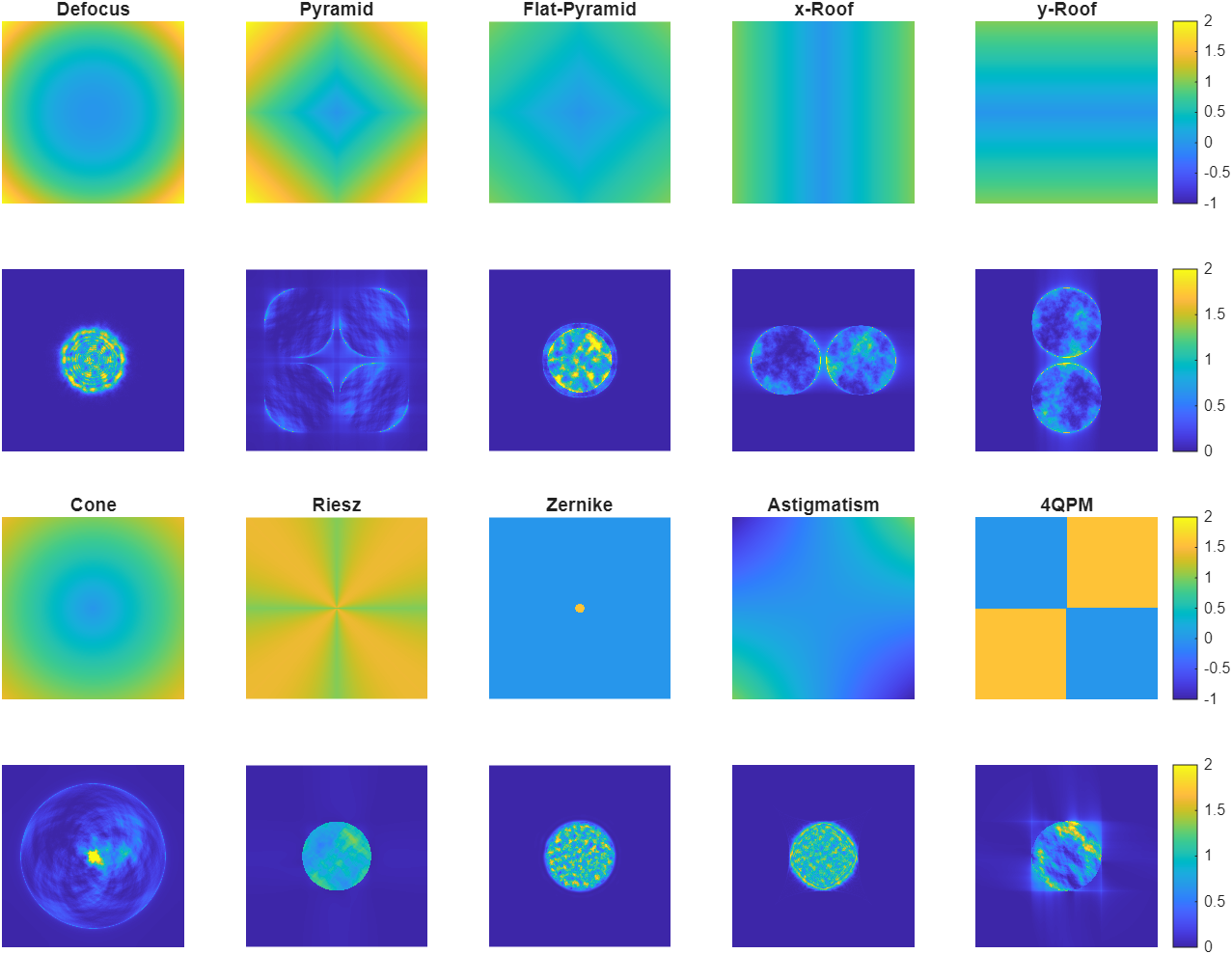}
    \caption{Illustration of shape functions $\Psi(\xiv)$, with $\xiv = (\xi_1,\xi_2)$, as defined in Table~\ref{table_psi} for different Fourier-type WFSs (1st and 3rd row). Below each function, the intensity image $I(\xv)$ as defined in \eqref{def_I_FtWFS} for some wavefront aberration $\phi(\xv)$  (2nd and 4th row).}
  \label{fig_otfs}
\end{figure}

As the name suggests, Fourier-type WFSs are typically modelled within the framework of Fourier optics \cite{Goodman_2005}. In particular, let $\phi : \R^2 \to \R$ denote the wavefront aberration and let $\chi_\Omega$ denote the indicator function of $\Omega \subset \R^2$ corresponding to the aperture of the imaging system. Furthermore, define the optical transfer function ($\OTF_\Psi$) by
    \begin{equation*}
        \OTF_\Psi(\xiv) := e^{i \Psi(\xiv)} \,, 
    \end{equation*}
where $\Psi(\xiv)$ is the shape function corresponding to the optical element or SLM \cite{Fauv16}; see Table~\ref{table_psi} and Figure~\ref{fig_otfs} for examples used below. With this, the intensity measurement of the Fourier-type WFS (recorded by the camera/detector) can be modelled as 
    \begin{equation}\label{def_I_FtWFS}
        I(\xv) =
        \abs{\FI\kl{ e^{i \Psi(\xiv)} \F\kl{\sqrt{n} \chi_\Omega e^{-i \phi}}(\xiv) }(\xv)}^2 \,.
    \end{equation}
where the constant $n > 0$ denotes the spatial average incoming flux \cite{Laidlaw_2025}. Note that the Fourier transforms model the passage of light through the two lenses, while the term $e^{i \Psi(\xiv)}$ corresponds to the modulation due to the optical element or SLM.

Concerning the reconstruction of $\phi$ from measurements of $I$, several reconstruction algorithms have been proposed in the past. These are typically based on simplifications of the nonlinear model \eqref{def_I_FtWFS} for specific forms of $\Psi(\xiv)$, and often use linearizations valid for slowly varying aberrations. For an overview of analytic reconstruction methods for the P-WFS see \cite{ShaHuRa20}, while for AI-based methods see, e.g., \cite{Vera_AO4ELT7proc}. For general choices of $\Psi(\xiv)$, the NOnlinear Pyramid Extension (NOPE) has been proposed in \cite{HuNeuSha_2023}. However, note that NOPE is currently too computationally expensive for real-time AO control.

% % % % % % % % % % % % % 
% Subsection - KEY IDEA %
% % % % % % % % % % % % %
\subsection{Key Ideas for Differential Fourier-type Wavefront Sensing}

In this section, we present our key ideas for connecting Fourier-type wavefront sensing with our above results on differential Fourier phase retrieval. For this, note first that the intensity measurement $I(\xv)$ defined in \eqref{def_I_FtWFS} does not depend on time; only a single measurement is made. However, we can artificially introduce time, and define
    \begin{equation}\label{def_Ixt_WFS}
        I(\xv,t) :=
        \abs{\FI\kl{ e^{i t \Psi(\xiv)} \F\kl{\sqrt{n}\chi_\Omega e^{-i \phi}}(\xiv) }(\xv)}^2 \,,
    \end{equation}
which we can equivalently rewrite as
    \begin{equation*}
        I(\xv,t) =
        \abs{\FI\kl{ e^{i t \Psi(\xiv)} u_\phi(\xiv)}(\xv)}^2 \,,
        \qquad
        \text{where}
        \qquad
        u_\phi(\xiv) := \F\kl{\sqrt{n} \chi_\Omega e^{-i \phi}}(\xiv)  \,.
    \end{equation*}  
Hence, $I(\xv,t)$ is of the form \eqref{def_Ixt} with $\psi(\xiv,t) = t\Psi(\xiv)$, and thus Theorem~\ref{thm_helper_general} yields
    \begin{equation*}
		\ddt I(\xv,t) = 2 \Re\kl{i A(\xv,t)e^{-i\vphi(\xv,t)} \G_\psi\kl{A(\cdot,t)e^{i\vphi(\cdot,t)}}(\xv)} \,,
	\end{equation*}
where we now have that
    \begin{equation}\label{def_GPsi}
        \G_\psi(v) = \FI\kl{ \psi'(t) \F(v)  }
        = \FI\kl{ \Psi \F(v)  } =: \G_\Psi(v) \,.
    \end{equation}
Next, note that while our actual Fourier-type WFS measurements $I(\xv)$ defined in \eqref{def_I_FtWFS} now correspond to the measurement $I(\xv,1)$, i.e., to $t=1$, we also have that
    \begin{equation*}
        I(\xv,0) =
        \abs{\FI\kl{ u_\phi(\xiv)}(\xv)}^2 
        =
        \abs{\sqrt{n} \chi_\Omega(\xv)}^2 
        =
        n \chi_\Omega(\xv)\,.
    \end{equation*}
Hence, in this specific setting we also have access to a second measurement, and thus
    \begin{equation*}
        \ddt I(\xv,t) \vert_{t=0} = \lim_{\tau \to 0} \frac{I(\xv,\tau) - I(\xv,0)}{\tau}
        \approx
        I(\xv,1) - I(\xv,0)
        =
        I(\xv) - n \chi_\Omega(\xv) 
    \end{equation*}	 
can be approximately computed numerically. Furthermore, note that since
    \begin{equation*}
        \FI\kl{ e^{i t \Psi(\xiv)} u_\phi(\xiv)}(\xv) \big\vert_{t=0}
        =
        \sqrt{n} \chi_\Omega(\xv) e^{-i \phi(\xv)} \,,
    \end{equation*}
it follows from the definition \eqref{polar_form_t} of $A(\xv,t)$ and $\vphi(\xv,t)$ that
    \begin{equation*}
        A(\xv,0) = \sqrt{n} \chi_\Omega(\xv) \,,
        \qquad
        \text{and}
        \qquad
        \vphi(\xv,0) = -\phi(\xv) \,.
    \end{equation*}
Combining the above arguments, we thus obtain the following result:
    
\begin{theorem}
Let $I(\xv,t)$ be defined as in \eqref{def_Ixt_WFS}, let $I(\xv) = I(\xv,1)$ be as in \eqref{def_I_FtWFS}, and let $\Psi(\xiv)$ be such $\Psi(\xiv) \F\kl{\chi_\Omega e^{-i \phi}}(\xiv) \in \LtRN$. Then $I(\xv,t) \in C^1([0,T],L^1(\RN))$ and
    \begin{equation}\label{main_WFS}
		I(\xv) - n \chi_\Omega(\xv)  
        \approx 
        \ddt I(\xv,t)\Big\vert_{t=0} = 2 n \Re\kl{i \, \chi_\Omega(\xv) e^{i\phi(\xv)} \G_\Psi\kl{\chi_\Omega e^{-i\phi}}(\xv)} \,.
	\end{equation}
\end{theorem}
\begin{proof}
This directly follows by combining Theorem~\ref{thm_helper_general} with the arguments above. Note only that for the definition space of $\G_\Psi : X \subset \LtRN \to \LtRN$ given in \eqref{def_GPsi}, we use
    \begin{equation*}
        X:= \Kl{ v \in \LtRN \, \vert \, \FI\kl{\Psi \F\kl{v}} \in \LtRN} \,,
    \end{equation*}
and that since $\Psi \F\kl{\chi_\Omega e^{-i \phi}} \in \LtRN$, Plancherel's theorem yields $\chi_\Omega e^{-i \phi} \in X$.
\end{proof}

Furthermore, if $\Psi(\xiv)$ is even or odd, then similarly as in Theorem~\ref{thm_main_wave}, we obtain

\begin{theorem}
Let $I(\xv,t)$ be defined as in \eqref{def_Ixt_WFS}, let $I(\xv) = I(\xv,1)$ be as in \eqref{def_I_FtWFS}, and let $\Psi(\xiv)$ be such $\Psi(\xiv) \F\kl{\chi_\Omega e^{-i \phi}}(\xiv) \in \LtRN$. Furthermore, assume that
     \begin{equation*}
        \chi_\Omega \cos(\phi) \in D(\G_\Psi) 
        \,, \qquad \text{and} \qquad
        \chi_\Omega \sin(\phi) \in D(\G_\Psi) \,,
    \end{equation*}
and that $\Psi(\xiv)$ is even, $\Psi(-\xiv) = \Psi(\xiv)$ for all $\xiv \in \RN$. Then there holds 
    \begin{equation}\label{eq_main_TIE_WFS}
    \begin{split}
		I(\xv) - n\chi_\Omega(\xv) 
        \approx  \ddt I(\xv,t)\Big\vert_{t=0} =&  -2 n\chi_\Omega(\xv)
        \sin\kl{\phi(\xv)}
        \G_\psi \kl{\chi_\Omega \cos\kl{\phi}}(\xv)
        \\
        &+
        2 n \chi_\Omega(\xv) \cos\kl{\phi(\xv)}\G_\psi \kl{\chi_\Omega\sin\kl{\phi}}(\xv) \,.
    \end{split}
	\end{equation}
If instead $\Psi(\xiv)$ is odd, i.e., if $\Psi(-\xiv) = -\Psi(\xiv)$ for all $\xiv \in \RN$, then
    \begin{equation}\label{eq_main_TPE_WFS}
    \begin{split}
		I(\xv) - n \chi_\Omega(\xv) 
        \approx  \ddt I(\xv,t)\Big\vert_{t=0} = &   2 n \chi_\Omega(\xv)
        \cos\kl{\phi(\xv)} 
        \G_\psi \kl{\chi_\Omega
        \cos\kl{\phi}}(\xv)
        \\
        & + 2 n \chi_\Omega(\xv) \sin\kl{\phi(\xv)}\G_\psi \kl{\chi_\Omega
        \sin\kl{\phi}}(\xv)\,.
    \end{split}
	\end{equation}   
\end{theorem}
\begin{proof}
This follows in the same way as in the proof of Theorem~\ref{thm_main_wave}. 
\end{proof}

Both \eqref{main_WFS}, and in particular \eqref{eq_main_TIE_WFS} and \eqref{eq_main_TPE_WFS} are (nonlinear) equation for the wavefront aberration $\phi(\xv)$, with their specific form depending on the concrete choice of $\Psi(\xiv)$. Hence, instead of reconstructing $\phi(\xv)$ from $I(\xv)$ using an iterative procedure such as NOPE, we can now determine it as the solution of an equation with right-hand side 
    \begin{equation}\label{helper_ddt_approx}
        I(\xv) - n \chi_\Omega(\xv) 
        \approx 
        \ddt I(\xv,t)\Big\vert_{t=0} \,. 
    \end{equation}
However, note that for the reconstruction to be accurate, the approximation \eqref{helper_ddt_approx} has to be accurate as well. Thus, in general the shape function $\Psi(\xiv)$ has to be small. E.g., we show below that this approach works well for a flat P-WFS, but fails for a steep one.

\begin{remark}
Similarly to before, it is possible to consider the above results for the case that the wavefront aberration $\phi(\xv)$ is sufficiently small such that $\sin(\vphi(\xv,t)) \approx \vphi(\xv,t)$ and $\cos(\vphi(\xv,t)) \approx 1$. This is commonly assumed in wavefront sensing in a closed-loop AO setting, when $\phi(\xv)$ is a small residual wavefront. In this case, \eqref{eq_main_TIE_WFS} simplifies to
    \begin{equation*}
    \begin{split}
		I(\xv) - n \chi_\Omega(\xv) 
        \approx   -2 n \chi_\Omega(\xv)
        \phi(\xv)
        \G_\psi \kl{\chi_\Omega }(\xv)
        +
        2 n \chi_\Omega(\xv) \G_\psi \kl{\chi_\Omega \phi}(\xv) \,,
    \end{split}
	\end{equation*}  
which we can rearrange into 
    \begin{equation}\label{approximation}
    \begin{split}
		I(\xv) - n \chi_\Omega(\xv) 
        \approx   
        2 n \kl{\rule{0pt}{2.5ex} \chi_\Omega(\xv) \G_\psi  
        -
        \G_\psi \kl{\chi_\Omega }(\xv) \mathcal{I} } 
        \kl{ \chi_\Omega(\xv) \phi(\xv) }\,.
    \end{split}
	\end{equation} 
Hence, if $\kl{\chi_\Omega(\xv) \G_\psi - \G_\psi \kl{\chi_\Omega }(\xv) \mathcal{I} }$ is continuously invertible, we obtain
    \begin{equation*}
    \begin{split}
		\chi_\Omega(\xv) \phi(\xv)   
        \approx   
        \frac{1}{2} \kl{\rule{0pt}{2.5ex} \chi_\Omega(\xv) \G_\psi  
        -
        \G_\psi \kl{\chi_\Omega }(\xv) \mathcal{I} } ^{-1}
        \kl{ \frac{I(\xv) - n \chi_\Omega(\xv)}{n} }\,.
    \end{split}
	\end{equation*}
which determines $\phi(\xv)$ on the aperture domain $\Omega$. Similarly, \eqref{eq_main_TPE_WFS} simplifies to
    \begin{equation*}
		I(\xv) - n \chi_\Omega(\xv) 
        \approx 
        2 n \chi_\Omega(\xv)
        \G_\psi \kl{\chi_\Omega}(\xv)
        + 2 n \chi_\Omega(\xv) \phi(\xv)\G_\psi \kl{\chi_\Omega
        \phi}(\xv)\,,
	\end{equation*}
which after rearranging again yields a nonlinear equation for $\phi(\xv)$ on $\Omega$, namely
    \begin{equation*}
		\chi_\Omega(\xv) \phi(\xv)\G_\psi \kl{\chi_\Omega
        \phi}(\xv)
        \approx
        \frac{1}{2}\kl{\frac{I(\xv) - n \chi_\Omega(\xv) }{n}}
        -
        \chi_\Omega(\xv)
        \G_\psi \kl{\chi_\Omega}(\xv)
        \,.
	\end{equation*}
\end{remark}

% % % % % % % % % % % %
% Subsection - G_\Psi %
% % % % % % % % % % % %
\subsection{\texorpdfstring{The Operator $\G_{\Psi}$}{The Operator G-Psi} for Different Wavefront Sensors}

As we saw above, our approach to Fourier-type WFSs based on differential intensity measurements leads to (nonlinear) equations for the wavefront aberration $\phi(\xv)$ involving the operator $\G_\psi$. Hence, we now derive some alternative characterizations of this operator for the shape functions $\Psi(\xiv)$ given in Table~\ref{table_psi}. Where this is easily possible, we directly consider the $N$-dimensional case, even though in practice $N=2$ is sufficient.
\begin{enumerate}
    \item The \textbf{Defocus WFS} is defined via the shape function
        \begin{equation}\label{def_Psi_Defocus}
            \Psi(\xiv) = \sum\limits_{k=1}^N c_k \, \xi_k^2 \,, 
		      \qquad \forall \, \xiv = (\xi_1\,,\dots\,,\xi_N) \in \RN \,.
        \end{equation}
    Using \eqref{Fourier_diff}, we can show as in Section~\ref{subsect_defoucs} (compare with \eqref{def_Pck}), that
        \begin{equation}\label{G_Psi_Defocus}
            (\G_\Psi v)(\xv) = \FI\kl{ \Psi \F(v)  }(\xv)
            = - \frac{1}{(2\pi)^2} \sum\limits_{k=1}^N c_k \dtdxkt v(\xv) \,,
        \end{equation}
    which is a well-defined element in $\LtRN$ for $v(\xv) \in H^2(\RN)$. 
    \item The \textbf{P-WFS} is defined via the shape function
        \begin{equation}\label{def_Psi_PWFS}
            \Psi(\xiv) = c \sum\limits_{k=1}^N \abs{\xi_k} \,,
            \qquad
            \forall \, \xiv = (\xi_1\,,\dots\,,\xi_N) \in \R^N \,.
        \end{equation}
    Using $\abs{\xi_k}  = \sgn(\xi_k) \xi_k$ and the expression \eqref{eq_Lapoh_Hilbert} applied to $\H_{x_k}$, we find that
        \begin{equation*}
            (\G_\Psi v)(\xv) 
            =
            c \sum\limits_{k=1}^N \FI\kl{ \abs{\xi_k}  \F(v)(\xiv)  }(\xv) 
            =
            \frac{c}{2\pi} \sum\limits_{k=1}^N  (\H_{x_k} \circ \partial_{x_k}) v(\xv)  \,,
        \end{equation*}
    which which is a well-defined element in $\LtRN$ for each $v \in H^1(\RN)$.
    \item The \textbf{$\xi_k$-Roof WFSs} are defined, for $k=1\,,\dots\,,N$, via the shape functions
        \begin{equation*}
            \Psi(\xiv) = c \abs{\xi_k} \,,
            \qquad
            \forall \, \xiv = (\xi_1\,,\dots\,,\xi_N) \in \R^N \,.
        \end{equation*}
    Analogously to the P-WFS, we thus find that
        \begin{equation*}
            (\G_\Psi v)(\xv)  
            =
            \frac{c}{2\pi}  (\H_{x_k} \circ \partial_{x_k}) v(\xv)  \,,
        \end{equation*}
    which is a well-defined element in $\LtRN$ for each $v \in H^1(\RN)$.
    \item The \textbf{Cone(s)-WFS} is defined via the shape function
        \begin{equation*}
            \Psi(\xiv) = c \abs{\xiv}^s 
            = c\kl{\sum_{k=1}^N \xi_k^2}^{s/2}\,, 
		      \qquad \forall \, \xiv = (\xi_1\,,\dots\,,\xi_N) \in \RN \,.
        \end{equation*}
    Together with the definition \eqref{def_fracLapl} of the fractional Laplacian, we thus find that
        \begin{equation*}
            (\G_\Psi v)(\xv) 
            =
            c \sum\limits_{k=1}^N \FI\kl{ \abs{\xiv}^s  \F(v)(\xiv)  }(\xv) 
            = 
            \frac{c}{(2\pi)^{s}} (-\Delta)^{s/2} v(\xv)\,,
        \end{equation*}
    which is a well-defined element in $\LtRN$ for each $v \in H^s(\RN)$. Note that for the classic \textbf{Cone-WFS} (corresponding to $s=1$), it follows with \eqref{helper_Riesz} that
        \begin{equation*}
            (\G_\Psi v)(\xv) 
            = 
            \frac{c}{(2\pi)} (-\Delta)^{1/2} v(\xv)
            =
            \frac{c}{(2\pi)} \div\kl{(\mathcal{R} v)(\xv)}\,.
        \end{equation*}
    For $s = 2$, the Cone(s)-WFS coincides with the Defocus-WFS. Furthermore, in the 1D case ($N=1$), the P-WFS and the Cone(s)-WFS with $s=1$ agree as well. 
    \item The \textbf{Riesz-WFS} is defined via the shape function
        \begin{equation*}
            \Psi(\xiv) = c \sum_{k=1}^N \frac{\abs{\xi_k}}{\abs{\xiv}} \,, 
		      \qquad \forall \, \xiv = (\xi_1\,,\dots\,,\xi_N) \in \RN \,.
        \end{equation*}
    Together with the definition \eqref{def_Riesz} of the Riesz transform, we thus find that
        \begin{equation*}
            (\G_\Psi v)(\xv) 
            =
            -c \sum\limits_{k=1}^N \FI\kl{ -i\frac{\xi_k}{\abs{\xiv}}(-i \sgn(\xi_k)) \F(v)(\xiv)  }(\xv) 
            = 
            c \sum\limits_{k=1}^N (\Hk \circ \Rk)v(\xv)\,,
        \end{equation*}
    which which is a well-defined element in $\LtRN$ for each $v \in \LtRN$. Note that similarly to the $\xi_k$-Roof WFS one could also define component-wise Riesz WFSs.
    \item The \textbf{Zernike-WFS} is defined via the shape function
        \begin{equation*}
            \Psi(\xiv) = c \,\chi_{B_\rho(0)}(\xiv) \,, 
		      \qquad \forall \, \xiv = (\xi_1\,,\dots\,,\xi_N) \in \RN \,,
        \end{equation*}
    where $B_\rho(0)$ denotes the unit-ball (in $\RN$) of radius $\rho$. Hence, with \eqref{Fourier_convolution} we get
        \begin{equation*}
            (\G_\Psi v)(\xv) 
            =
            c \, \FI\kl{ \chi_{B_\rho(0)} \F(v)(\xiv)  }(\xv) 
            = 
            c \kl{ \FI\kl{ \chi_{B_\rho(0)}} \ast v}(\xv)\,,
        \end{equation*}
    which together with the well-known identity
        \begin{equation*}
            \FI\kl{ \chi_{B_\rho(0)}}(\xv)
            =
            \rho^\frac{N}{2} \abs{\xv}^{-\frac{N}{2}} J_{N/2}(2\pi \rho \abs{\xv}) \,,
        \end{equation*}
    where $J_{N/2}(x)$ denotes the Bessel function of the 1st kind of order $N/2$, yields
        \begin{equation*}
            (\G_\Psi v)(\xv) 
            =
            c \rho^\frac{N}{2} \kl{ \kl{ \abs{\cdot}^{-\frac{N}{2}} J_{N/2}(2\pi \rho \abs{\cdot})} \ast v}(\xv)\,,
        \end{equation*}
    \vspace{10pt}
    Note that $(\G_\Psi v)(\xv)$ is a well-defined element in $\LtRN$ for each $v \in \LtRN$.
    \item The \textbf{Astigmatism-WFS} is defined via the shape function (in 2D, i.e., $N=2$)
        \begin{equation}\label{def_Psi_Astig}
            \Psi(\xiv) = c \, \xi_1 \, \xi_2 \,, 
		      \qquad \forall \, \xiv = (\xi_1\,,\xi_2) \in \R^2 \,.
        \end{equation}
    Using \eqref{Fourier_diff}, we can show as in the case of the Defocus-WFS that
        \begin{equation*}
            (\G_\Psi v)(\xv) = c \, \FI\kl{ \xi_1 \, \xi_2 \F(v)(\xiv)  }(\xv)
            =  \frac{c}{(2\pi)^2} \partial_{x_1} \partial_{x_2} v(\xv) 
            =  \frac{c}{(2\pi)^2} \partial_{x_2} \partial_{x_1} v(\xv) \,,
        \end{equation*}
    which is a well-defined element in $\LtRN$ for $v(\xv) \in H^2(\RN)$.
    \item The \textbf{4QPM-WFS} is defined via the shape function (in 2D, i.e., $N=2$)
        \begin{equation}\label{def_Psi_iQuad}
            \Psi(\xiv) 
            =
            \begin{cases}
                c \,, & \xi_1 \, \xi_2 \leq 0 \,,
                \\
                0 \,, & \text{else} \,,
            \end{cases}
		      \qquad \qquad 
              \forall \, \xiv = (\xi_1\,,\xi_2) \in \R^2 \,,
        \end{equation}
    which can be rewritten as $\Psi(\xiv) = c\, \sgn(\xi_1)\,\sgn(\xi_2)$. Together with \eqref{eq_Hk_sgn}, we obtain 
        \begin{equation*}
            (\G_\Psi v)(\xv) = c \, \FI\kl{ (-i\sgn(\xi_1))\,(-i\sgn(\xi_2)) \F(v)(\xiv)  }(\xv)
            = \H_{x_1} \H_{x_2} v(\xv) \,.
        \end{equation*}
    Note that $(\G_\Psi v)(\xv)$ is a well-defined element in $\LtRN$ for each $v \in \LtRN$. Again, similarly to the $\xi_k$-Roof WFS, component-wise 4QPM-WFSs are possible.
\end{enumerate}

At this point, we have to discuss the indicator function $\chi_\Omega(\xv)$ in \eqref{def_I_FtWFS}. On the one hand, it ensures that $\chi_\Omega(\xv) e^{-i \phi(\xv)} \in \LtRN$, and thus that $I(\xv) \in L^1(\RN)$ is well-defined. On the other hand, if $D(\G_\Psi) \subset H^s(\Omega)$ for $s \geq 1/2$, as is the case for the Defocus, Pyramid, Roof, Cone, and Astigmatism WFS, then $\G_\Psi(\chi_\Omega )$ is not well-defined. In practice, this does not pose any substantial difficulties, since the phase $\phi(\xv)$ is (and can only be) reconstructed within the aperture domain $\Omega$ anyways, where $\chi_\Omega(\xv)$ is constant. However, it does pose theoretical obstacles, which have already been encountered in previous analyses of Fourier-type WFSs \cite{HuNeuSha_2023}. There, it was proposed to replace $\chi_\Omega(\xv)$ in \eqref{def_I_FtWFS} by a compactly supported and smooth approximation $\chi^\eps_\Omega(\xv)$ which satisfies $\chi^\eps_\Omega(\xv) = 1$ for all $\xv \in \Omega$. While this is certainly possible here as well, it also introduces a number of technicalities obstructing the main purpose of the subsequent section. Hence, and since we have already used the approximation \eqref{helper_ddt_approx}, for the following application to specific Fourier-type WFSs we simply assume that $\Omega = \RN$, and interpret (Fourier) transforms such as $\F\kl{\chi_\Omega e^{-i \phi}}(\xiv) = \F\kl{ e^{-i \phi}}(\xiv) $ in the distributional sense. Where appropriate, we then also state the corresponding equations which would have been obtained using the smooth approximation $\chi^\eps_\Omega(\xv)$.

% % % % % % % % % % % % % % % % % % % % % % %
% Subsection - Defocus and Astigmatism WFSs % 
% % % % % % % % % % % % % % % % % % % % % % %
\subsection{Applications, Simplifications, Reconstruction Formulas}

In this section, we derive explicit PDEs for the aberration $\phi(\xv)$ for the Defocus and Astigmatism WFS, and provide an explicit reconstruction algorithm for all Fourier-type WFSs with even shape functions $\Psi(\xiv)$ based on the approximation of a small $\phi(\xv)$. 

First, consider the Defocus WFS defined via \eqref{def_Psi_Defocus}, for which we have shown in \eqref{G_Psi_Defocus} that $\G_\Psi$ is a second order differential operator as in \eqref{def_Pck}. Hence, analogously to \eqref{eq_TIE_generalized}, and recalling our assumption $\Omega = \RN$, we find that in this case \eqref{main_WFS} takes the form
    \begin{equation*}
		\kl{\frac{I(\xv) - n}{n}} \approx - \frac{1}{(2\pi)^2} \nabla_{\xv} \cdot \kl{ \Cb  \nabla_{\xv} \phi(\xv)} \,.
	\end{equation*}
where $\Cb := \operatorname{diag}(2c_1\,,\dots\,,2c_N) \in  \R^{N\times N}$. If $c_k = c$ for all $k$, this further simplifies to
    \begin{equation}\label{eq_defocus}
        \Delta \phi(\xv) \approx -\frac{(2\pi)^2}{2c } \kl{\frac{I(\xv) - n}{n}} \,,
	\end{equation}
which can easily be solved (numerically) for the aberration $\phi(\xv)$. Note the conceptual and mathematical similarity of the D-WFS and \eqref{eq_defocus} to the curvature WFS \cite{Roddier_1988}.   

\begin{remark}
Note that the shape function $\Psi(\xiv)$ of the P-WFS given in \eqref{def_Psi_PWFS} can, for small values of $c$, be reasonably well approximated by the shape function of the D-WFS; see also the similar intensities in Figure~\ref{fig_otfs}. Hence, for measurements obtained by a flat P-WFS, \eqref{eq_defocus} can be used to efficiently reconstruct the wavefront aberration $\phi(\xv)$.
\end{remark}

In the case of the Astigmatism WFS defined via \eqref{def_Psi_Astig}, we similarly obtain the PDE
    \begin{equation}\label{eq_astig}
         \partial_{x_2} \partial_{x_1} \phi(\xv) = -\frac{(2\pi)^2}{2c} \kl{\frac{I(\xv) - n}{n}} \,,
    \end{equation}
a hyperbolic equation which can again be easily solved for the aberration $\phi(\xv)$.

\begin{remark}
Note that when using the smooth approximation $\chi^\eps_\Omega(\xv)$ of the indicator function $\chi_\Omega(\xv)$, a close inspection of the derivation of \eqref{eq_defocus} and \eqref{eq_astig} shows that they remain valid in the interior of $\Omega$, leaving them suitable for wavefront reconstruction.
\end{remark}

For general Fourier-type WFSs characterized by even shape functions $\Psi(\xiv)$, consider again the approximation \eqref{approximation} for $\sin(\vphi(\xv,t)) \approx \vphi(\xv,t)$ and $\cos(\vphi(\xv,t)) \approx 1$, i.e.,  
    \begin{equation*}
		I(\xv) - n \chi_\Omega(\xv)  
        \approx   
        2 n \kl{\rule{0pt}{2.5ex} \chi_\Omega(\xv) \G_\psi  
        -
        \G_\psi \kl{\chi_\Omega }(\xv) \mathcal{I} } 
        \kl{ \chi_\Omega(\xv) \phi(\xv) }\,.
	\end{equation*}
Under our current assumption that $\Omega = \RN$, this further simplifies to
    \begin{equation*}
		\kl{I(\xv) - n}
        \approx   
        2 n \kl{\rule{0pt}{2.5ex}  \G_\psi  
        -
        \G_\psi \kl{1}(\xv) \mathcal{I} } 
        \phi(\xv) \,.
	\end{equation*}
Now note that if $\Psi(\xiv)$ is continuous in zero, then $\G_\Psi(1) = \Psi(0)$, and thus
    \begin{equation}\label{FWFS_MAIN_forward}
		\kl{I(\xv) - n}
        \approx   
        2 n \kl{\rule{0pt}{2.5ex}  \G_\psi  
        -
        \Psi(0) \mathcal{I} } 
        \phi(\xv) \,.
	\end{equation}
Hence, if $\kl{\G_\psi-\Psi(0) \mathcal{I} } $ is invertible, then we obtain the reconstruction formula
    \begin{equation}\label{FWFS_MAIN}
		\phi(\xv)
        \approx 
        \frac{1}{2}
        \kl{\rule{0pt}{2.5ex}  \G_\psi  
        -
        \Psi(0) \mathcal{I} }^{-1}\kl{\frac{I(\xv) - n}{n}} 
        =
        \frac{1}{2}
        \FI\kl{ \frac{1}{\Psi(\xiv)- \Psi(0)}\F \kl{\frac{I(\cdot) - n}{n}} }(\xiv) 
         \,,
	\end{equation}
which can be efficiently implemented via the fast (discrete) Fourier transform. Note that if $\Psi(\xv)$ is not continuous in zero, then $\G_\Psi(1)$ needs to be computed appropriately. For example, $\G_\Psi(1) = c/2$ for the 4QPM sensor \eqref{def_Psi_iQuad}, which is discontinuous in $0$.

\begin{remark}
Note that for most of our considered WFSs, there holds $\Psi(\xiv) = 0$. This is in particular true for the Defocus and Astigmatism WFSs, for which  \eqref{FWFS_MAIN_forward} then coincides with \eqref{eq_defocus} and \eqref{eq_astig}, respectively. Finally, note that a small regularization parameter may be added in \eqref{FWFS_MAIN} to ensure that the denominator stays bounded away from zero.
\end{remark}

% % % % % % % % % % % % % % % % % % %
% Section - Numerical Exaperiments  %
% % % % % % % % % % % % % % % % % % %
\section{Numerical Experiments}\label{sect_numerics}

In this section, we present a number of numerical experiments showcasing the practical usefulness of our derived results. In particular, we first consider phase reconstructions in the case of quadratic modulations $\psi(\xiv,t)$ via the generalized TIE \eqref{eq_TIE_generalized}, and then investigate the reconstruction quality of \eqref{FWFS_MAIN} for different Fourier-type WFSs.

In both cases, we consider a piecewise constant discretization of the involved functions on an equidistant $N \times N$ (pixel) grid centered around the origin, and, with a small abuse of notation, identify the discretized functions with their continuous counterparts. With this choice, the continuous Fourier transform can be suitably discretized via the discrete Fourier transform (DFT), which in turn can be efficiently implemented via the (inverse) fast Fourier transform (\texttt{(i)fft2}). All computations presented below were carried out in Matlab R2025a on a standard desktop computer.

% % % % % % % % % % % % % % % % %
% Subsection - Results General  %
% % % % % % % % % % % % % % % % %
\subsection{Numerical Results: Quadratic Modulation}

\begin{figure}[ht!]
    \centering
    \includegraphics[width=\linewidth]{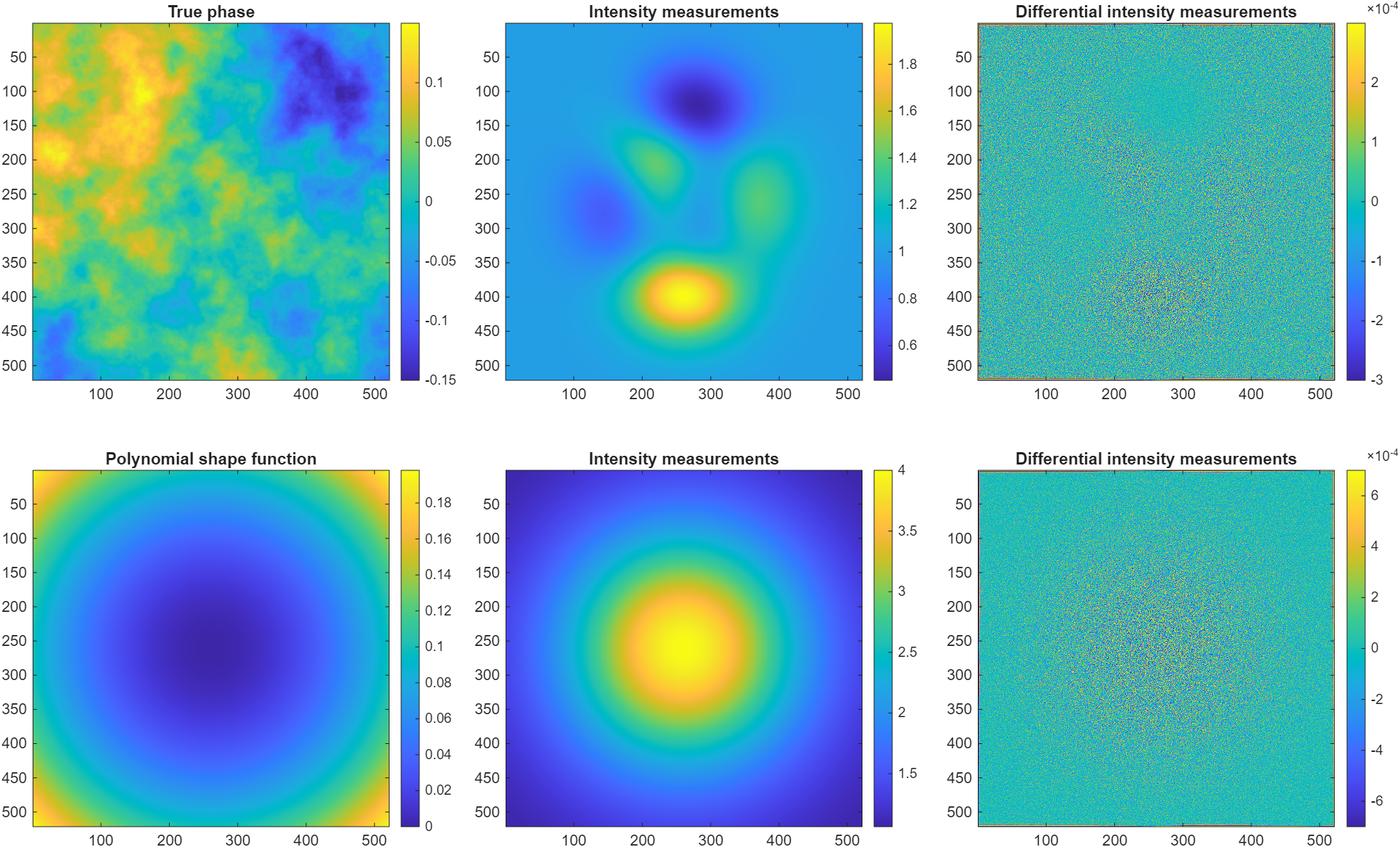}
    \caption{Datasets for quadratic modulation experiment: ground truth phase $\vphi(\xv,0)$ (top left), quadratic modulation function $\psi(\xiv,1)$ (bottom left), example intensities $I
    (x,1)$ (middle) and approximate differential intensity measurements $\Delta I(\xv)$ (right) for two different test settings, i.e., $A(\xv,0)$ \texttt{peaks} or Gaussian (top and bottom, resp.).}
    \label{fig:dataset_poly}
\end{figure}

\noindent
In our first experiment, we consider the setting of Section~\ref{subsect_defoucs}, i.e., the reconstruction of the phase $\vphi(\xv,0)$ in \eqref{polar_form_t} from measurements of $I(\xv,t)$ in the case of a quadratic modulation $\psi(\xiv,t) = c t (2\pi^2) \abs{\xiv}^2$. In this case, $\vphi(\xv,0)$ satisfies the TIE \eqref{eq_TIE_classic}, i.e.,
    \begin{equation*}
		\nabla_{\xv} \cdot \kl{ I(\xv,0) \, \nabla_{\xv} \vphi(\xv,0)} =
        \frac{1}{c}\ddt I(\xv,t) \Big\vert_{t=0}
        \approx 
        \frac{1}{c} \kl{I(\xv,1)-I(\xv,0)}\,.
	\end{equation*}
In order for the last approximation to be sufficiently accurate, we set $c=10^{-2}$ in the definition of $\psi(\xiv,t)$. In order to solve the above equation for $\vphi(\xv,0)$, we adapt the Fourier multiplier approach from \eqref{TIE_FM}, which in our discretized setting now reads
    \begin{equation*}
        \vphi(\xiv,0)
        \approx 
        \texttt{ifft2}\kl{ g_\alpha\kl{2\pi i\abs{\xiv}}\texttt{fft2} \kl{\frac{1}{2c I(\xv,0)}
        \texttt{ifft2}\kl{ g_\alpha\kl{2\pi i\abs{\xiv}} \texttt{fft2} \kl{\Delta I}}}(\xiv)}
\end{equation*}
where
    \begin{equation}\label{gal}
        \Delta I(\xv) := I(\xv,1) - I(\xv,0)\,,
        \qquad
        \text{and} 
        \qquad
        g_\alpha(\lambda) := \frac{\lambda}{\lambda^2+\alpha} \,,
    \end{equation}
is the Tikhonov regularization filter \cite{Engl_Hanke_Neubauer_1996} with regularization parameter $\alpha = 10^{-14}$, which is used to avoid potential zero divisions in \eqref{TIE_FM}. For the amplitude $A(\xv,0)$, we use two different choices, a Gaussian bump function and a scaled version of Matlab's \texttt{peaks}, both shifted away from zero. The ground truth phase $\vphi(\xv,0)$, quadratic modulation $\psi(\xiv,1)$, the corresponding intensities $I(\xv,1)$, computed via \eqref{polar_form_t}, and the resulting approximate differential intensity measurements $\Delta I(\xv)$ are depicted in Figure~\ref{fig:dataset_poly}. 

\begin{figure}[ht!]
    \centering
    \includegraphics[width=\linewidth]{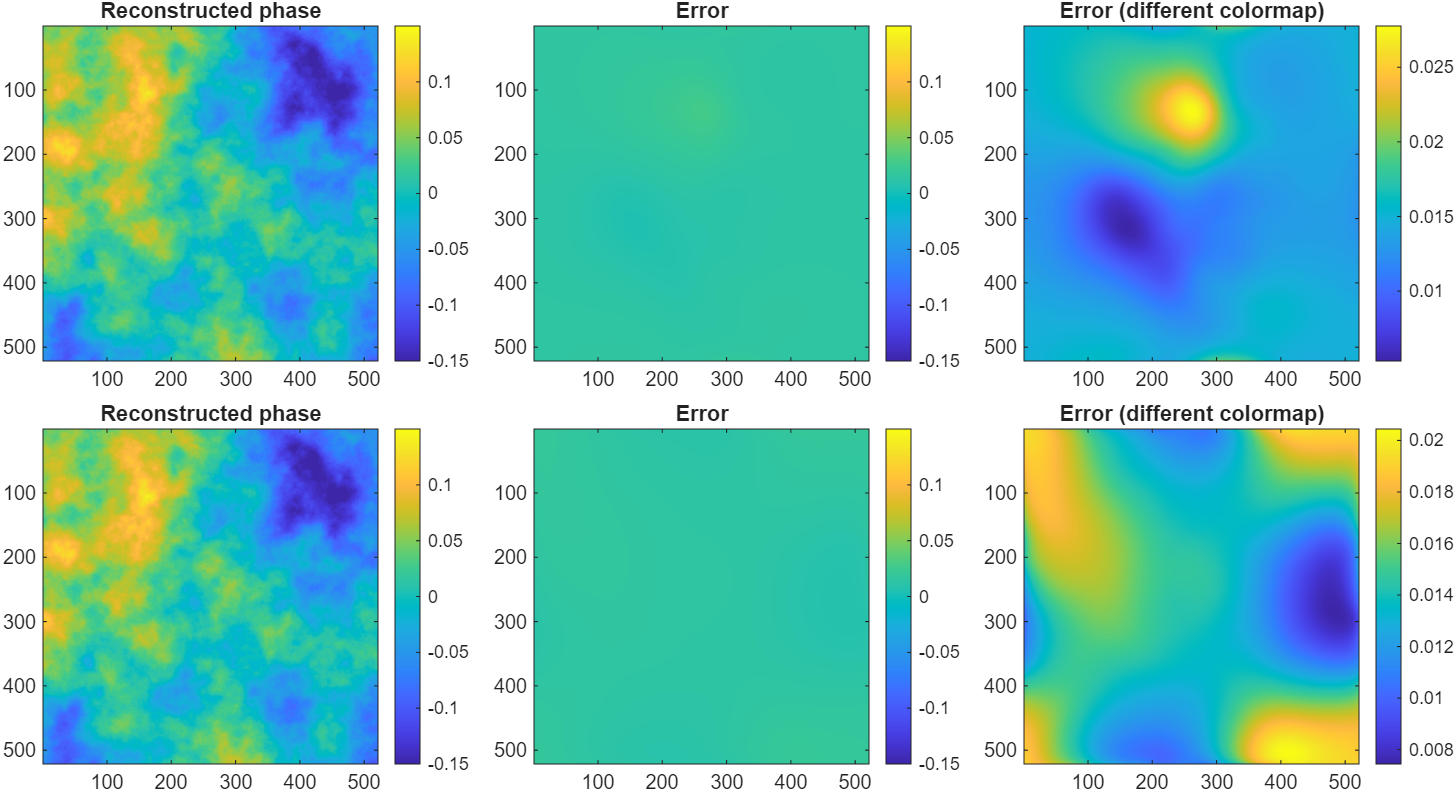}
    \caption{Phase reconstructions in the quadratic modulation experiment: reconstructed phase $\vphi(\xv,0)$ (left) for \texttt{peaks} (top) and Gaussian amplitude $A(\xv,0)$ (bottom). Middle row: reconstruction error displayed in colormap scale of the reconstruction. Right row: reconstruction error in tight colormap scaled to the actual error range.}
    \label{fig:rec_poly}
\end{figure}

Figure~\ref{fig:rec_poly} presents the reconstructions obtained with the approach discussed above, as well as the reconstruction error in two different colormaps for enhanced visibility. In both cases, we observe that the reconstructions are very close to the ground truth phase, which is particularly evident in the difference plots on the right of the figure.

% % % % % % % % % % % % % % %
% Subsection - Results WFSs %
% % % % % % % % % % % % % % %
\subsection{Numerical Results: Fourier-type WFSs}

\begin{figure}[ht!]
    \centering
    \includegraphics[width=\linewidth]{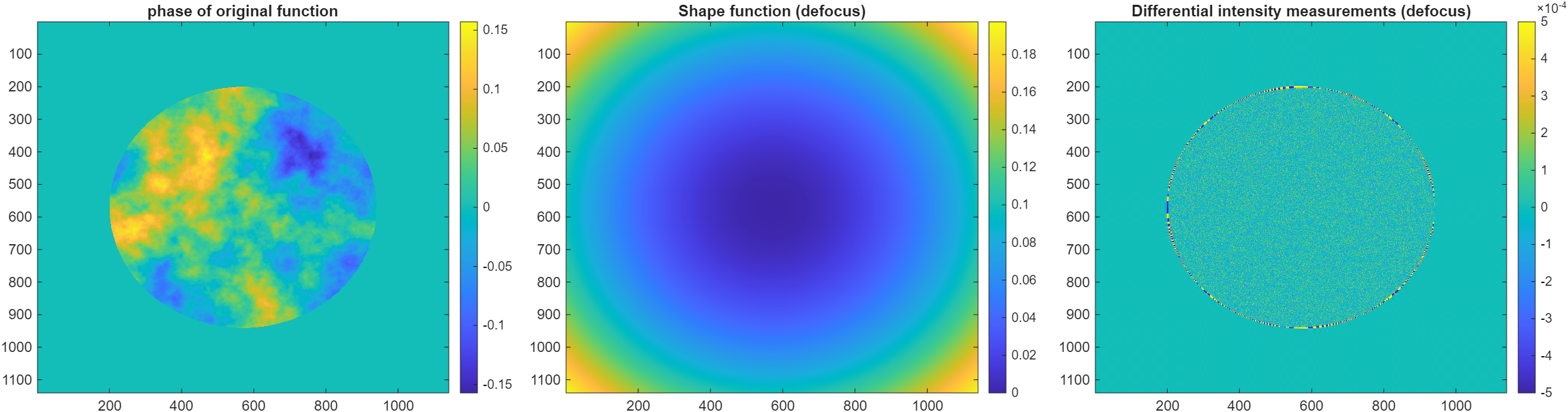}
    \caption{Dataset for Fourier-type WFS experiments: ground truth wavefront phase $\phi(\xv)$ (left), defocus shape function $\Psi(\xv)$ (middle) and corresponding approximate differential intensity measurements $(I(\xv) - 1)$ (right).}
    \label{fig:dataset_WFR}
\end{figure}

\noindent
In our second experiment, we investigate the reconstruction quality of the reconstruction formula \eqref{FWFS_MAIN} for different Fourier-type WFSs, which in our discrete setting reads
    \begin{equation}\label{rec_WFS_finite}
        \phi(\xv)
        =
        \frac{1}{2}
        \texttt{ifft2}\kl{ g_\alpha\kl{\Psi(\xiv)- \Psi(0)}\texttt{fft2} \kl{\frac{I - n}{n}}(\xiv) }
         \,,
    \end{equation} 
where $g_\alpha(\lambda)$ is the Tikhonov regularization filter of \eqref{gal}, again used with $\alpha = 10^{-14}$. Note that as in the continuous setting, if $\Psi$ is not continuous in zero, $\Psi(0)$ has to be replaced by $\G_\Psi(1)$. Furthermore, note that by using this formula, we implicitly assume that $\Omega = \R^2$, i.e., that the aperture is the whole space. Since this is unrealistic in practice, we choose the ground truth wavefront $\phi(\xv)$ such that it is zero outside of a circular domain mimicking a telescope aperture, and then restrict our reconstruction to this domain. Without loss of generality, we also set the spatial average incoming flux $n=1$, which only amounts to a simple scaling. Figure~\ref{fig:dataset_WFR} depicts the ground truth wavefront $\phi(\xv)$, the shape function $\Psi(\xv)$ of the Defocus WFS, as well as the corresponding differential intensity measurements $(I(\xv)-1)$. Note that in all the results presented below, we have used the choice $c=10^{-2}$ in the definition of the shape functions $\Psi(\xiv)$; cf.~Table~\ref{table_psi}. Hence, the intensities $I(\xv)$ generally look different than those depicted in Figure~\ref{fig_otfs}, where much larger values of $c$ were chosen for illustration.

\begin{figure}[ht!]
    \centering
    \includegraphics[width=\linewidth]{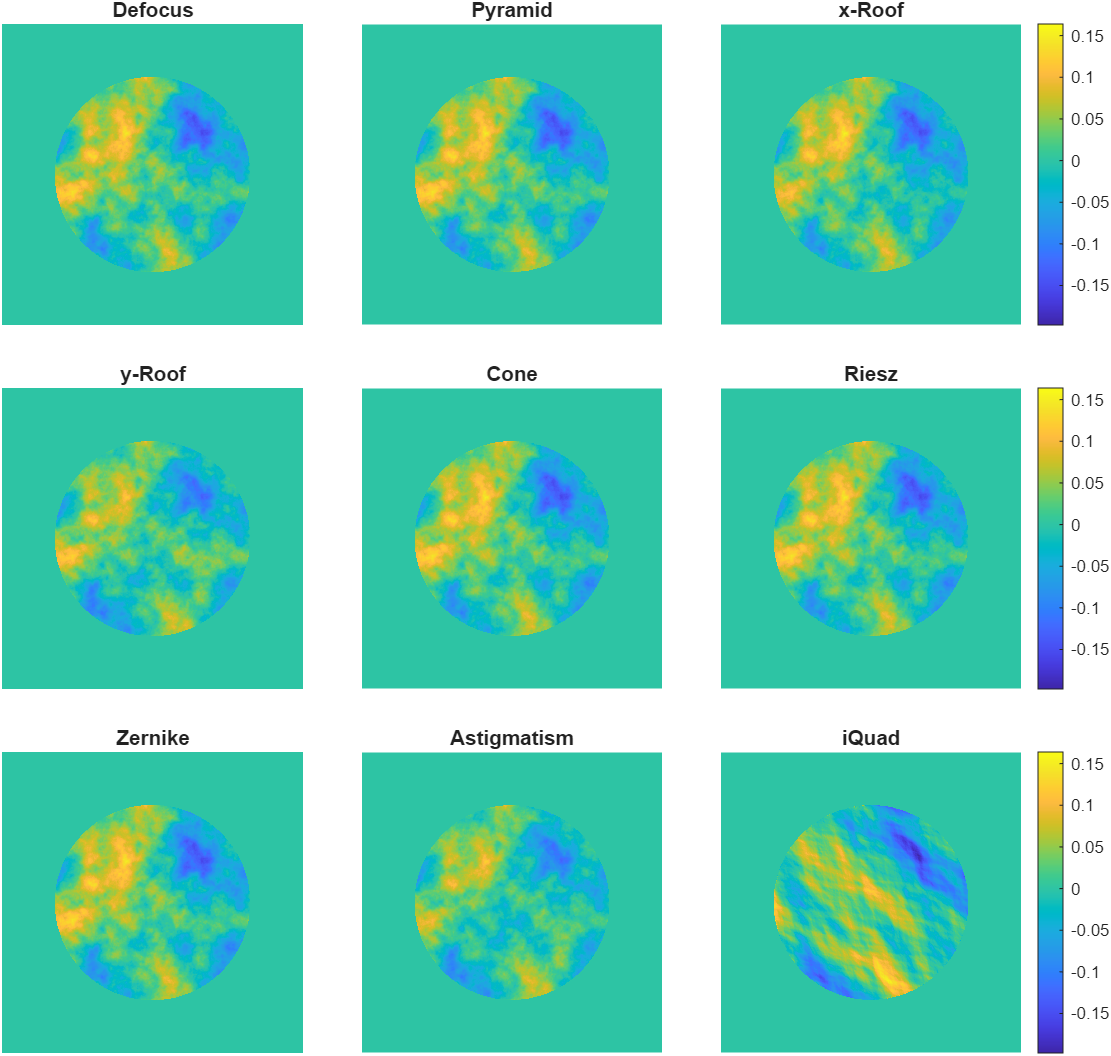}
    \caption{Fourier-type WFS experiments: wavefront reconstructions obtained via the reconstruction formula \eqref{rec_WFS_finite} for different shape-functions $\Psi(\xiv)$ as given in  Table~\ref{table_psi}.}
    \label{fig:rec_WFR}
\end{figure}

\begin{figure}[ht!]
    \centering
    \includegraphics[width=\linewidth]{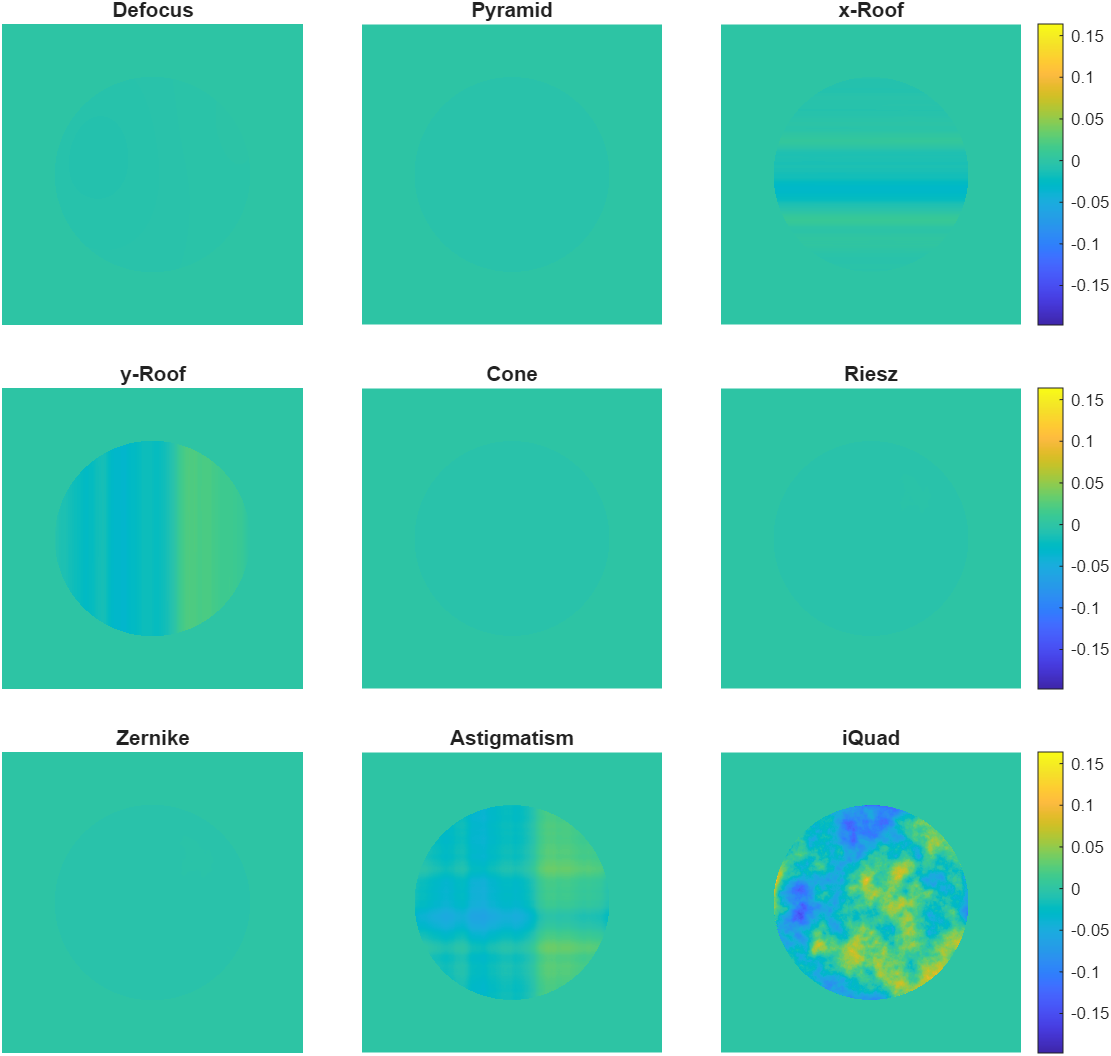}
    \caption{Fourier-type WFS experiments: reconstruction errors obtained via the reconstruction formula \eqref{rec_WFS_finite} for different shape-functions $\Psi(\xiv)$ as given in  Table~\ref{table_psi}.}
    \label{fig:error_WFR}
\end{figure}

Figure~\ref{fig:rec_WFR} depicts the wavefronts $\phi(\xv)$ reconstructed via \eqref{rec_WFS_finite} for the different shape functions $\Psi(\xiv)$ given in Table~\ref{table_psi}. With the exception of the 4QPM-WFS, the reconstructions for all WFSs are visually very close to the ground truth wavefront $\phi(\xv)$. This is mostly confirmed via the reconstruction error plots depicted in Figure~\ref{fig:error_WFR}, which reveal only minor errors for the x- and y-Roof, as well as the Astigmatism WFSs. Overall, the reconstruction quality appears to be excellent, which is promising for future tests of our reconstruction formula \eqref{rec_WFS_finite} in a full AO simulation environment.

% % % % % % % % % % % % % % % % % % %
% Section - Conclusion and Outlook  %
% % % % % % % % % % % % % % % % % % %
\section{Conclusion}\label{sect_conclusion}

In this paper, we derived a series of partial differential and integro-differential equations for the unknown phase $\vphi(\xv,t)$ in the context of Fourier phase retrieval from differential intensity measurements. We showed how these equations can be used as the basis for the design and implementation of efficient phase reconstruction algorithms, and derived an explicit wavefront reconstruction formula for the phase retrieval problem underlying wavefront sensing via Fourier-type WFSs. Numerical experiments then illustrated the usefulness of our proposed approaches for both phase retrieval and wavefront sensing.

% % % % % % % % % % %
% Section - Support %
% % % % % % % % % % %
\section{Acknowledgments \& Support}

This research was funded in part by the Austrian Science Fund (FWF) SFB 10.55776/F68 ``Tomography Across the Scales'', project F6805-N36 (Tomography in Astronomy) and project F6807-N36 (Tomography with Uncertainties). For open access purposes, the authors have applied a CC BY public copyright license to any author-accepted manuscript version arising from this submission. The first author thanks Christina Strohmenger (University of Vienna) and Stefan Kindermann (JKU Linz) for fruitful discussions.

% % % % % % % % %
% Bibliography  %
%% % % % % % % % %
\bibliographystyle{plain}
{\footnotesize
\bibliography{mybib,bib_fourier_clear}
}

\end{document}

%% file: header.tex
\usepackage{amsmath}
\usepackage{amsthm}
\usepackage{amssymb}
\usepackage{cite}
\usepackage{url}
\usepackage{footmisc}
\usepackage{graphicx}
\usepackage{epstopdf}
\usepackage{chngcntr}
\usepackage{enumitem}
\usepackage{hhline}
\usepackage{array}
\usepackage{hyperref}
\usepackage{color}
\usepackage{multirow}
\usepackage{tabularx}
\usepackage{tikz}

\hypersetup{colorlinks=false}

\makeatletter
\let\@fnsymbol\@arabic
\makeatother

\tikzset{every picture/.style={line width=0.75pt}} 
\usetikzlibrary{patterns}

\tikzset{
    pattern size/.store in=\mcSize, 
    pattern size = 5pt,
    pattern thickness/.store in=\mcThickness, 
    pattern thickness = 0.3pt,
    pattern radius/.store in=\mcRadius, 
    pattern radius = 1pt
}
\makeatletter
\pgfutil@ifundefined{pgf@pattern@name@_7gn8i7dj0}{
    \pgfdeclarepatternformonly[\mcThickness,\mcSize]{_7gn8i7dj0}
    {\pgfqpoint{0pt}{0pt}}
    {\pgfpoint{\mcSize+\mcThickness}{\mcSize+\mcThickness}}
    {\pgfpoint{\mcSize}{\mcSize}}
    {
    \pgfsetcolor{\tikz@pattern@color}
    \pgfsetlinewidth{\mcThickness}
    \pgfpathmoveto{\pgfqpoint{0pt}{0pt}}
    \pgfpathlineto{\pgfpoint{\mcSize+\mcThickness}{\mcSize+\mcThickness}}
    \pgfusepath{stroke}
    }
}
\makeatother

\if@twoside
\else 
\fi

\numberwithin{equation}{section}
\counterwithin{figure}{section}
\counterwithin{table}{section}

\theoremstyle{plain}% default
\newtheorem{theorem}{Theorem}[section]
\newtheorem{lemma}[theorem]{Lemma}
\newtheorem{proposition}[theorem]{Proposition}

\newtheorem{corollary}[theorem]{Corollary}

\newtheorem{assumption}{Assumption}[section]

\theoremstyle{definition}

\newtheorem{example}{Example}[section]

\theoremstyle{remark}
\newtheorem*{remark}{Remark}

\newcommand{\norm}[1]{\left\|#1\right\|}
\newcommand{\abs}[1]{\left\vert#1\right\vert}

\newcommand{\kl}[1]{\left(#1\right)}
\newcommand{\Kl}[1]{\left\{#1\right\}}

\newcommand{\R}{\mathbb{R}} 
\newcommand{\C}{\mathbb{C}}
\newcommand{\N}{\mathbb{N}}

\def\div{\mfunc{div}} % redef!!!

\newcommand{\sgn}{\mathrm{sgn}}

\newcommand{\eps}{\varepsilon}

\newcommand{\ddt}{\frac{\partial}{\partial t}}
\newcommand{\ddx}{\frac{\partial}{\partial x}}
\newcommand{\ddxk}{\frac{\partial}{\partial x_k}}
\newcommand{\dtdxkt}{\frac{\partial^2}{\partial x_k^2}}

\newcommand{\F}{\mathcal{F}}
\newcommand{\FI}{\mathcal{F}^{-1}}
\newcommand{\G}{\mathcal{G}}

\renewcommand{\H}{\mathcal{H}}
\newcommand{\HI}{\mathcal{H}^{-1}}
\newcommand{\Hk}{\mathcal{H}_{x_k}}
\newcommand{\Rk}{\mathcal{R}_k}
\newcommand{\xv}{\boldsymbol{x}}
\newcommand{\xiv}{\boldsymbol{\xi}}
\newcommand{\yv}{\boldsymbol{y}}

\newcommand{\Lcat}{\mathcal{L}_{c_\alpha(t)}}
\newcommand{\Scat}{\mathcal{S}_{c_\alpha(t)}}
\renewcommand{\L}{\mathcal{L}}
\renewcommand{\S}{\mathcal{S}}

\renewcommand{\div}{\operatorname{div}}

\newcommand{\vphi}{\varphi}

\newcommand{\LtR}{{L_2(\R)}}

\newcommand{\Rt}{{\R^2}}
\newcommand{\ca}{c_{\alpha}}

\newcommand{\Pcat}{\mathcal{P}_{c_\alpha(t)}}

\newcommand{\Pckt}{\mathcal{P}_{c_k(t)}}
\newcommand{\LtRN}{{L_2(\R^N)}}
\newcommand{\RN}{{\R^N}}
\newcommand{\Cb}{\boldsymbol{C}}

\newcommand{\OTF}{\operatorname{OTF}}